\documentclass[11pt]{article}

\usepackage{amssymb}   
\usepackage{amsthm}    
\usepackage{amsmath}   
\usepackage{stmaryrd}  
\usepackage{titletoc}  
\usepackage{mathrsfs}  
\usepackage{graphicx}
\usepackage{color, enumerate}
\newlength{\defbaselineskip}
\newcommand{\setlinespacing}[1]                                     %
           {\setlength{\baselineskip}{#1 \defbaselineskip}}

\theoremstyle{plain}
\newtheorem{thm}{Theorem}[section]
\newtheorem{cor}[thm]{Corollary}
\newtheorem{lem}[thm]{Lemma}
\newtheorem{prop}[thm]{Proposition}

\theoremstyle{definition}
\newtheorem{defn}{Definition}[section]
\newtheorem{ass}{Assumption}[section]
\newtheorem{rmk}{Remark}[section]

\newcommand{\eps}{\varepsilon}

\newcommand{\cP}{\mathcal{P}}
\newcommand{\cH}{\mathcal{H}}
\newcommand{\cL}{\mathcal{L}}
\newcommand{\cM}{\mathcal{M}}

\newcommand{\cD}{\mathcal{D}}

\newcommand{\cO}{\mathcal{O}}
\newcommand{\cS}{\mathcal{S}}

\newcommand{\cZ}{\mathcal{Z}}

\newcommand{\bH}{\mathbb{H}}
\newcommand{\bP}{\mathbb{P}}
\newcommand{\bR}{\mathbb{R}}
\newcommand{\bN}{\mathbb{N}}

\newcommand{\sF}{\mathscr{F}}

\newcommand{\sL}{\mathscr{L}}
\newcommand{\sS}{\mathscr{S}}
\newcommand{\sB}{\mathscr{B}}

\newcommand{\sM}{\mathscr{M}}
\newcommand{\sP}{\mathscr{P}}

\newcommand{\rrow}{\rightarrow}

\DeclareMathOperator*{\esssup}{ess\, sup}

\DeclareMathOperator*{\essinf}{ess\, inf}

\makeatletter\@addtoreset{equation}{section} \makeatother

 \allowdisplaybreaks

\begin{document}

\title{Optimal Liquidation in Target Zone Models and Neumann Problem of Backward SPDEs with Singular Terminal Condition\footnotemark[1]}
\author{Robert Elliott\footnotemark[2] \quad \quad Jinniao Qiu\footnotemark[2]  \quad \quad Wenning Wei\footnotemark[2] }

\footnotetext[1]{This work was partially supported by the National Science and Engineering Research Council of Canada and by the start-up funds from the University of Calgary. }
\footnotetext[2]{Department of Mathematics \& Statistics, University of Calgary, 2500 University Drive NW, Calgary, AB T2N 1N4, Canada. \textit{E-mail}: \texttt{relliott@ucalgary.ca} (R. Elliott), \texttt{jinniao.qiu@ucalgary.ca} (J. Qiu),  \texttt{wenning.wei@ucalgary.ca} (W. Wei).}

\maketitle

\begin{abstract}
We study the optimal liquidation problems in target zone models using dynamic programming methods. Such control problems allow for stochastic differential equations with reflections and random coefficients. The value function is characterized with a Neumann problem of backward stochastic partial differential equations (BSPDEs) with singular terminal conditions. The existence and the uniqueness of strong solution to such BSPDEs are addressed, which in turn yields the optimal feedback control. In addition, the unique existence of strong solution to Neumann problem of general semilinear BSPDEs in finer functions space, a comparison theorem, and a new link between forward-backward stochastic differential equations and BSPDEs are proved as well.

\end{abstract}



\section{Introduction}
Let $(\Omega, \bar{\sF}, (\bar{\sF}_t)_{t\in[0,T]},\bP)$ be a complete filtered probability space with $(\bar{\sF}_t)_{t\in[0,T]}$ being the augmented filtration generated by an independent point process $\tilde{J}$ on a non-empty Borel set $\mathcal{Z}\subset\bR^l$ with finite characteristic measure $\mu(dz)$ and two independent Wiener processes $W$ and $B$. The set $\mathcal{Z}$  is endowed with its Borel $\sigma$-algebra $\mathscr{Z}$, and the associated Poisson random measure is denoted by $\pi(dt,dz)$. Throughout this paper, we denote by $(\sF_t)_{t\in[0,T]}$ the augmented filtration generated by $W$. The predictable $\sigma$-algebras on $\Omega\times[0,T]$ corresponding to $(\sF_t)_{t\in[0,T]}$ and  $(\bar{\sF}_t)_{t\in[0,T]} $ are denoted by $\sP$ and $\bar{\sP}$, respectively.

The concerned optimal liquidation in target zone models may be described as a stochastic optimal control problem as follows: for $q\in (1,\infty)$,
\begin{align}
\min_{(\xi,\rho)\in \mathscr A} 
E\left[\int_0^T\left(\eta_s(y^{0,y}_s)|\xi_s|^q +\lambda_s(y^{0,y}_s)|x^{0,x}_s|^q\right)\,ds
+\int_0^T\int_{\mathcal {Z}}\gamma_s(y^{0,y}_s,z)|\rho_s(z)|^q\,\mu(dz)ds
 \right], \label{OCP}
\end{align}
subject to
\begin{equation}\label{controlled-states}
\left\{
\begin{split}
x_r^{0,x}&=x-\int_0^r\xi_s\,ds -\int^r_0\int_ {\mathcal {Z}}\rho_s(z)\,\pi(dz,ds),\quad r\in[0,T],\\
x_T^{0,x}&=0,\\
y_r^{0,y}&=y+\int_0^r\beta_s(y^{0,y}_s)\,ds+\int_0^r\sigma_s(y^{0,y}_s)\,dW_s+\int_0^r\bar{\sigma}_s(y^{0,y}_s)\,dB_s+L_r,\quad r\in[0,T],\\
y^{0,y}_r&\geq a, \text{a.s. for all }r\in[0,T].\\
\int_0^T&(y^{0,y}_s-a)\,dL_s=0, \quad \text{(Skorohod condition)}
\end{split}
\right.
\end{equation}
where the Wiener processes $W$ and $B$ have dimensions $d$ and $m$ respectively. The controlled real-valued \textit{state process} $(x_t)_{t\in[0,T]}$  describes the number of assets/securities held at time $t \in [0,T]$  in a portfolio liquidation framework, and it is governed by a pair of \textit{controls} $(\xi,\rho)$ that represent the rates at which the portfolio is liquidated in the primary market and the block trades are placed, for instance, in the dark pools, respectively, with the Poisson random measure $\pi$ governing dark pool executions.  The set of \textit{admissible controls}, denoted by $\mathscr A$, consists of all pairs $(\xi,\rho)\in\cL^q_{\bar{\sF}}(0,T; \bR)\times \cL^{q}_{\bar{\sF}}(0,T;L^q(\mathcal{Z}))$ ($q\in(1,\infty)$) satisfying almost surely the \textit{terminal state constraint}
\begin{equation} \label{terminal-condition}
	x_{T}=0.
\end{equation}

The real-valued \textit{uncontrolled} process $(y_t)_{t\in[0,T]}$, also called \textit{factor process}, is satisfying a stochastic differential equations (SDE) reflected from below, with the possibly random and nonlinear coefficients $\beta_t(y;\omega), \bar \sigma_t(y;\omega)$ and $\sigma_t(y;\omega)$ being $\sF$-adapted. Such reflected processes have often been proposed as models (for instance, currency exchange rates) in target zones; see \cite{bertola1992target,krugman1991target,neuman2016optimal} for instance. As discussed in \cite{neuman2016optimal} for optimal portfolio liquidation in target zone models, we use the process $(y_t)_{t\in[0,T]}$ to model the price evolution of the holding assets/securities. In this paper, we shall use $x^{s,x,\xi,\rho}_t$ and $y^{s,y}_t$ for $0\leq s\leq t\leq T$ to indicate the dependence of the state process on the control $(\xi,\rho)$, the initial time $s \in [0,T]$ and initial states $x,y\in \mathbb{R}$.

Define the dynamical cost function
\begin{align*}
J_t(x,y;\xi,\rho)=E\left[\int_t^T\left(\eta_s(y^{t,y}_s)|\xi_s|^q +\lambda_s(y^{t,y}_s)|x^{t,x,\xi,\rho}_s|^q\right)\,ds \right. \\
+\int_t^T\int_{\mathcal {Z}}\gamma_s(y^{t,y}_s,z)|\rho_s(z)|^q\,\mu(dz)ds
\left.~\Big|\bar\sF_t\right], \quad t\in[0,T],
\end{align*}
where the coefficients $\eta_s(y)$, $\lambda_s(y)$ and $\gamma_s(y,z)$ are $\sF$-adapted.
The value function is given by
\begin{align}
V_t(x,y) =\essinf_{(\xi,\rho)\in\mathscr A}J_t(x,y;\xi,\rho)\quad t\in[0,T). \label{value-func}
\end{align}
In a portfolio liquidation problem, the terms associated with coefficients $\eta_t(y;\omega)$ and $\lambda_t(y;\omega)$ measure the market impact costs and the investor's desire for early liquidation (``risk aversion''), respectively, while the term associated with $\gamma_t(y;\omega)$ denotes the so-called \textit{slippage} or \textit{adverse selection costs} associated with the execution of dark pool orders; see \cite{GraeweHorstQui13} for instance.
 The value function $V_t(x,y)$ measures the cost of liquidating the portfolio comprising $x$ shares during the time interval $[t,T]$, given the current value $y$ of the factor process, and the terminal constraint (\ref{terminal-condition}) reflects the fact that full liquidation is required by the terminal time.

Models of optimal portfolio liquidation \textit{without} target zones have been extensively studied in the mathematical finance and stochastic control literature in recent years; see, e.g., \cite{AlmgrenChriss00,AnkirchnerKruse12,Forsyth2012,GatheralSchied11,GraeweHorstQui13,HorstNaujokat13,Horst-Qiu-Zhang-14,Kratz13,KratzSchoeneborn13,kruse2016minimal,Schied13}. By contrast, the optimal liquidation in target zone models has just caused an attention only recently. By means of catalytic superprocesses, Neuman and Schied \cite{neuman2016optimal} studied a class of optimal liquidation problems in target zone models, which are different from our concerned problem \eqref{OCP} in the following five respectives: (i) the terminal state constraint like \eqref{terminal-condition} is not attached, i.e., the full liquidation is not required in \cite{neuman2016optimal}; (ii) the optimization therein is over strategies that only trade when the price process is located at the barrier, and it does not allow block trades, while in our liquidation problem \eqref{OCP} we do not consider trading at the barrier but allow both the continuous trades in the primary market and the block trades in dark pools; (iii) the corresponding power $q$ in \cite{neuman2016optimal} is restricted in $[2,\infty)$; (iv) the (reflected) price processes are of Markovian type, while the coefficients in problem \eqref{OCP} may be random and thus the price process $y^{0,y}_{\cdot}$ may not be Markovian; (v) the corresponding stochastic control problem in \cite{neuman2016optimal} was solved by means of a scaling limit of critical branching particle systems, also known as a catalytic superprocess, whereas in this work, we use the dynamic programming methods. A more recent work by Belak, Muhle-Karbe and Ou \cite{belak2018optimal} also shows some results on optimal liquidation in target zone models, which is concerned with a class of linear quadratic cases without terminal state constraint \eqref{terminal-condition}.

In this paper, we use the general dynamic programming principle for controlled SDEs with random coefficients (see \cite{Peng_92,qiu2017viscosity}). Formally, we may derive the corresponding stochastic Hamilton-Jacobi-Bellman equation, which together with the $q$th-power structure of the cost functional further suggests a multiplicative decomposition of the value function of the form
\begin{equation} \label{ansatz}
  V_t(x,y)=u_t(y)|x|^q \quad \textrm{and}\quad \Psi_t(x,y)=\psi_t(y)|x|^q,
\end{equation}
for a pair of adapted processes $(u,\psi)$ that satisfies the following backward stochastic partial differential equation (BSPDE): 
\begin{equation}\label{ValueFunction_u}
\left\{
\begin{split}
-du_t(y)&=\bigg[\alpha  D^2u +\sigma^* D\psi 
+\beta  Du +\lambda  -\frac{|u|^{q^*}}{(q^*-1)|\eta|^{q^*-1}}-\mu(\cZ) u
\\
&\quad \quad
+\int_{\mathcal{Z}}\frac{\gamma_{\cdot}(\cdot,z) u}{(|\gamma_{\cdot}(\cdot,z)|^{q^*-1}+|u |^{q^*-1})^{q-1}}\,\mu(dz)\bigg](t,y)\,dt
-\psi_t(y) dW_t,\\
&\quad \quad  (t,y)\in[0,T)\times\cD,\\
Du_t(a)&=0,\quad t\in[0,T).\\
u_T(y)&=\infty, \quad  y\in\bar{\cD},
\end{split}
\right.
\end{equation}
where $q^*=\frac{q}{q-1}$ is the H$\ddot{\text{o}}$lder conjugate of $q$ and
$$
\alpha_t(y):=\frac{1}{2}\big[\sigma^*_t(y)\sigma_t(y)+\bar\sigma^*_t(y)\bar\sigma_t(y)\big].
$$

The preceding BSPDE has a nonlinear growth on $u_t(y)$ and is endowed with a Neumann boundary condition and a singular terminal value. To the best of our knowledge, such a Neumann problem of BSPDEs has never been studied before, even though BSPDEs have been extensively studied in the applied probability and financial mathematics literature; see, e.g., \cite{bender2016-first-order,Bensoussan-fbsd-1983,Bayraktar-Qiu_2017,DuTangZhang-2013,EnglezosKaratzas09,Hu_Ma_Yong02,Mania-Tevzaze-2003}. In fact, the Neumann problem of BSPDEs has just caused an attention recently. Bayraktar and Qiu \cite{Bayraktar-Qiu_2017} obtained the existence and uniqueness of strong solutions for certain types of Neumann problems of BSPDEs on bounded domains and under standard Lipschitz assumptions. However, the methods adopted in \cite{Bayraktar-Qiu_2017} are not applicable to BSPDE \eqref{ValueFunction_u} because of the nonlinear growth, unbounded domain and the singular terminal condition. 

In this work, we prove the existence and uniqueness of the strong solution to BSPDE \eqref{ValueFunction_u} from which the optimal control is derived via the verification theorem. To construct and verify the optimal control requires the composition of the strong solution and the factor process $(y_t^{0,y})_{t\geq 0}$. However, the insufficient regularity of the strong solution prevents us from using the existing It\^o-Kunita-Wentzell formula established in \cite[Lemmas 4.1]{Bayraktar-Qiu_2017}. In fact, this difficulty motivates us to prove the verification theorem by means of a link between BSPDEs and forward-backward stochastic differential equations (FBSDEs). Nevertheless, our proof for the link requires certain regularity of the solution including the \textit{boundedness} of its gradient. To ensure the boundedness of gradients, the unique existence of strong solution is established in function spaces (see $\cM^1$ in Section 2) which are finer than those in \cite{Bayraktar-Qiu_2017}. We first prove the existence and uniqueness of solutions in space $\cM^1$ for Neumann problems of general semilinear BSPDEs,  a comparison theorem is also established under the boundedness assumption on gradients of solutions, and then we show that a solution to the BSPDE with singular terminal value may be obtained as the limit of a sequence of solutions to BSPDEs with finite terminal values. It is worth noting that the bounded estimate of gradients is derived from estimates of the approximating sequence and the existing maximum principles for weak solutions of quasilinear BSPDEs with Dirichlet boundary conditions (see \cite{fu2017maximum,QiuTangMPBSPDE11}). Finally, the uniqueness of strong solution is obtained via the verification theorem and the comparison theorem. 


	The remainder of this paper is organized as follows. Our main assumptions and results are summarized in Section 2. Section 3 is devoted to the proof of the link between FBSDEs and BSPDEs. In Section 4, we prove the existence and uniqueness of strong solution for Neumann problems of general semilinear BSPDEs as well as a comparison theorem. The existence and uniqueness of strong solution to BSPDE \eqref{ValueFunction_u} and the verification theorem are established in Section 5. Finally, we recall in the appendix the generalized It\^o-Wentzell formula with a corollary and a maximum principle for weak solutions of quasilinear BSPDEs in general domains.

\section{Preliminaries and main result}
We first introduce some notations. Set $\cD=[a,+\infty)$ and denote by $H^{m,p}(\cD)$ the space of all the functions on $\cD$ with up to $m$th-derivatives in $L^p(\cD)$ for $p\in[1,\infty]$. We write $H^m(\cD)$ instead of $H^{m,2}(\cD)$ for simplicity, and we use $\langle\cdot,\,\cdot\rangle$ and $\|\cdot\|$ to denote respectively the inner product and norm in the usual Hilbert space $L^2(\cD):=H^{0,2}(\cD)$.

For $p\in[1,\infty]$ and a Banach space $\bH$ with norm $\|\cdot\|_{\bH}$, the space $L^p(\Omega,\sF_T;\bH)$ is the set of all $\bH$-valued $\sF_T$-measurable and $L^p$-integrable random variables. For $0\leq s\leq t\leq T$,  we denote by $\sS^p_{\sF} (s,t;\bH)$ the set of all the $\bH$-valued and $\sF_r$-adapted continuous processes $(X_{r})_{r\in [s,t]}$ such that
\[
	\|X\|_{\sS _{\sF}^p(s,t;\bH)}:= \left\|  \sup_{r\in [s,t]} \|X_r\|_{\bH} \right\|_{L^p(\Omega)}< \infty.
\]
With a subscript, we define $\sS^{p}_{w,\sF} (s,t;\bH)$ as the space of  all the $\bH$-valued and $\sF_r$-adapted \textit{weakly} continuous processes\footnote{This means that for any $\phi$ in $\bH^*$ (the dual space of $\bH$), the mapping $r \mapsto \phi(X_r)$ is a.s. continuous on $[s,t]$.} $(X_{r})_{r\in [s,t]}$, equipped with the same norm: 
$\|\cdot\|_{\sS _{w,\sF}^p(s,t;\bH)}= \|\cdot\|_{\sS _{\sF}^p(s,t;\bH)}$. 
By $\sL^p_{\sF}(s,t;\bH)$, we denote the class of $\bH$-valued $\sF_r$-adapted processes $(u_r)_{r\in[s,t]}$ such that
\begin{align*}
\|u\|_{\sL^p_{\sF}(s,t;\bH)} &:=\left\| \|u(\cdot)\|_{\bH}    \right\|_{L^p(\Omega\times[s,t])}  <\infty.
\end{align*}
For $p\in[1,\infty)$,  the space $\sM^p_{\sF}(s,t;\bH)$ is the set of all the $\bH$-valued processes $(u_t)_{t\in[s,t]}$ belonging to $\sL^p_{\sF}(s,t;\bH)$ and satisfying
\begin{align*}
\|u\|_{\sM^p_{\sF}(s,t;\bH)} &:=\left(\textrm{esssup}_{\omega\in \Omega}\sup_{r\in[s,t]}E\left[\int_r^t\|u(\omega ,\tau , \cdot)\|^p_{\bH}\,d\tau | \mathcal{F}_r\right]\right)^{1/p}<\infty.
\end{align*}
  In a similar way, we define $\sS _{\bar{\sF}}^p(s,t;\bH)$, $\sL^p_{\bar{\sF}}(s,t;\bH)$ and $\sM^p_{\bar{\sF}}(s,t;\bH)$, and all these defined spaces of processes are complete.

%

For $k=0$ or $1$, and for a nonempty domain $\cO\subseteq\cD$,  denote 
\begin{align*}
&\cH^k([s,t]\times\cO):=\left(\sS^2_{\sF}(s,t;H^k(\cO))\cap\sL^2_{\sF}(s,t;H^{k+1}(\cO))\right)\times\sL^2_{\sF}(s,t;H^{k}(\cO)),
\\
&\cM^k([s,t]\times\cO):=\left(\sS^{\infty}_{\sF}(s,t;H^k(\cO))\cap\sM^2_{\sF}(s,t;H^{k+1}(\cO))\right) \times\sM^2_{\sF}(s,t;H^{k}(\cO)),
\end{align*} 
equipped with norms:
\begin{align*}
\|(u,\psi)\|_{\cH^{k}([s,t]\times \cO)}
&=\|u\|_{\sS^{2}_{\sF}(s,t;H^k(\cO))} + \|u\|_{\sL^2_{\sF}(s,t;H^{k+1}(\cO))} + \|\psi\|_{\sL^2_{\sF}(s,t;H^{k}(\cO))},
\\
\|(u,\psi)\|_{\cM^{k}([s,t]\times \cO)}
&=\|u\|_{\sS^{\infty}_{\sF}(s,t;H^k(\cO))} + \|u\|_{\sM^2_{\sF}(s,t;H^{k+1}(\cO))} + \|\psi\|_{\sM^2_{\sF}(s,t;H^{k}(\cO))}.
\end{align*}
Obviously, we have $\cM^k([s,t]\times\cO) \subset \cH^k([s,t]\times\cO)$. For simplicity, we write $\cH^k=\cH^k([0,T]\times\cD)$ and $\cM^k=\cM^k([0,T]\times\cD)$.  


\begin{defn} \label{defn-soltn}
A pair of processes $(u,\psi)$ is a strong solution to equation \eqref{ValueFunction_u} if for all $\tau\in(0,T)$ and $b\in\mathbb R$ with $b>a$, it holds that $(u,\psi)1_{[0,\tau]\times [a,b]}\in\cH^1([0,\tau]\times[a,b])$, and with probability 1, for all $t\in[0,\tau]$,
\begin{align*}
u_t(y)
&=u_{\tau}(y)
+\int_t^{\tau}\bigg[\alpha  D^2u
+\sigma^*D\psi+\beta Du+\lambda
-\frac{|u|^{q^*}}{(q^*-1)|\eta|^{q^*-1}}-\mu(\cZ) u
\\
&\quad\quad
+\int_{\mathcal{Z}}\frac{\gamma_{\cdot}(\cdot,z) u}{(|\gamma_{\cdot}(\cdot,z)|^{q^*-1}+|u |^{q^*-1})^{q-1}}\,\mu(dz)
\bigg](s,y)\,ds
 -\int_t^{\tau}\psi_s(y) dW_s,\quad \text{dy-a.e.},
\end{align*}
with
$$Du _t(a)=0, \text{ for } t\in [0,\tau],\quad \text{ and } \lim_{\tau\rrow T}u_{\tau}(y)=\infty, \quad \text{for all }y\in \cD, \text{ a.s.}$$
\end{defn}
We would note that the zero Neumann boundary condition $Du _t(a)=0$ is holding in the sense that for each $t\in[0,T)$,
$$\lim_{\delta\rightarrow 0^+}\frac{1}{\delta}\int_a^{a+\delta}Du_t(x)\, dx=0,\quad \text{a.s.}$$

The main results below are established under the following assumptions on the coefficients.
\begin{ass}\label{AssHJB}

\begin{enumerate}[$({\mathcal A} 1)$]
	\item (Measurability and boundedness) The function $\gamma: \Omega\times[0,T]\times\bR\times\mathcal{Z}\longrightarrow [0,+\infty]$
is $\sP\times \mathscr{B}(\bR)\times \mathscr{Z}$-measurable, and the functions
   $$\beta,\sigma,\bar{\sigma},\eta,\lambda: \Omega\times[0,T]\times\bR\longrightarrow 
   \bR\times \bR^d\times \bR^{m}\times \bR_+\times \bR_+$$
are $\sP\times \mathscr{B}(\bR)$-measurable and essentially bounded by $\Lambda>0$. 
\item (Lipschitz-continuity)
For $h=\lambda,\eta,\beta,\sigma^i,\bar\sigma^j$, $i=1,\dots,d$, $j=1\dots,m$,  it holds that for all $y_1,y_2\in \bR$ and $(\omega,t)\in \Omega\times [0,T]$,
\begin{align*}
	\left|h_t(y_1)-h_t(y_2)\right|  + \esssup_{z\in \cZ} \left|  \gamma_t(y_1,z)-\gamma_t(y_2,z)    \right|
			\leq \Lambda \, \left|y_1-y_2\right|,
\end{align*}
where $\Lambda$ is the constant in $(\mathcal A 1)$.
\item 
There exist constants $\kappa>0$ and $\kappa_0>0$ such that $\eta_s(y)\geq \kappa_0$ and
$$
\text{(Superparabolicity)}\quad\quad
\sum_{i=1}^m  \left|  \bar{\sigma}^{i}_s(y)\right|^2  \geq \kappa,
\quad\text{a.s.},\quad \forall\,(s,y)\in  [0,T]\times \bR.
$$
\end{enumerate}

\end{ass}

To the best of our knowledge, there is no existing $L^{p}$-theory for Neumann problems of BSPDEs for any $p\in[1,2)\cup(2,\infty)$. At the same time, in order to derive the representation for the composition $u_t(y_t)$,   the solution $u$ to BSPDE \eqref{ValueFunction_u} has to be regular enough to allow for an application of the link between FBSDEs and BSPDEs that is to be established. To guarantee  the regularity, we need to develop the theory of strong solutions to BSPDEs with Neumann boundary conditions, working with a weighted solution. Throughout this paper, the weight function is chosen (not uniquely) to be the following one:
\[
	\theta: \mathbb{R} \to \mathbb{R}, \quad y \mapsto \left(1+|y-a|^2\right)^{-1},
\]
and we may analyze ${\theta}u$ instead of $u$. A direct computation verifies that $(u,\psi)$ is a solution to \eqref{ValueFunction_u} if and only if $(v,\zeta):=({\theta}u,{\theta}\psi)$ solves
\begin{equation}\label{HJBequ_weight}
\left\{
\begin{split}
-dv_t(y)&=\bigg[\alpha D^2v+\sigma^*D\zeta+\lambda\theta
-\frac{|v|^{q^*}}{(q^*-1)|\theta \eta|^{q^*-1}}-\mu(\cZ) v
\\
&\quad\quad
+\int_{\mathcal{Z}}\frac{\theta \gamma(\cdot,z) v}{(|\theta \gamma(\cdot,z)|^{q^*-1}+|v |^{q^*-1})^{q-1}}\,\mu(dz)
+f(t,y,Dv,v,\zeta)\bigg](t,y)\,dt\\
&\quad\quad-\zeta_t(y)\,dW_t,\quad (t,y)\in(0,T)\times\cD,\\
D v_t(a) &=0, \quad \text{for }t\in[0,T),\\
v_T(y)&=\infty,
\end{split}
\right.
\end{equation}
with the linear term
\begin{align*}
&f(t,y,Dv,v,\zeta)\\
&= \left[\beta+4(y-a)\alpha_t(y)\theta\right]Dv
+2\theta\left[\alpha_t(y)+(y-a)\beta\right]v+2(y-a)\theta\sigma^*\zeta\\
&=\left[\beta+4(y-a) \alpha \theta\right]Dv
+2\theta\left[   \alpha     +(y-a)\beta\right]v+2(y-a)\theta\sigma^*\zeta.
\end{align*}

The main results include a link between FBSDEs and BSPDEs in Section 3, the unique existence of strong solution in $\cM^1$ and a comparison principle for Neumann problems of general semilinear BSPDEs under standard Lipschitz conditions in Section 4, and the well-posedness of BSPDE \eqref{ValueFunction_u} which together with the solvability of the optimal liquidation problem \eqref{OCP} is summarized below for the reader's convenience.
\begin{thm}\label{main-results}
Let Assumption \ref{AssHJB} hold. We assert that:
\begin{enumerate}[(i)]
\item (Existence of strong solution).
 The BSPDE \eqref{ValueFunction_u}  admits a strong solution $(u,\psi)$ such that $ (\theta u,\theta\psi)\in \cM^1([0,\tau]\times \cD)$, $D(\theta u) \in \sS^{\infty}_{w,\sF}(0,\tau;L^{\infty}(\cD))$ for $\tau\in[0,T)$, and 
$$
\frac{c_0 }{(T-t)^{q-1}}
	\leq u_t(y) \leq
		\frac{C_0}{(T-t)^{q-1}},\quad\text{a.s. } \forall\,(t,y)\in [0,T)\times \cD,
$$
where the positive constants $c_0>0$ and $ C_0>0$ depend only on $q,\kappa_0, \Lambda, T$ and $\mu(\cZ)$.
\item (Verification theorem).
 For the above strong solution $(u,\psi)$, the random field
\begin{equation*}
    V(t,y,x):=u_t(y)|x|^q,\quad (t,x,y)\in[0,T]\times \bR\times \cD,
\end{equation*}
coincides with the value function of \eqref{value-func}. Moreover, the optimal (feedback) control is given by
\begin{equation*} 
  \left(\xi^{*}_t,\, \rho^*_t(z)\right)= \left(\frac{\left|u_t(y_t)\right|^{q^*-1} x_t}{\left| \eta_t(y_t)\right|^{q^*-1}},\, \frac{\left|u_t(y_t)\right|^{q^*-1}x_{t-}}{\left|\gamma_t(y_t,z)\right|^{q^*-1}+\left|u_t(y_t)\right|^{q^*-1}}   \right),
  \quad \text{for }t\in[0,T).
  \end{equation*}
  \item (Uniqueness)
   If $(\tilde{u},\tilde{\psi})$ is another strong solution of \eqref{ValueFunction_u} satisfying
 $$
(\theta\tilde{u},\theta\tilde{\psi}+\sigma^*D(\theta\tilde{\psi}))\in
  \cH^1([0,t]\times \cD)
  \quad\text{and } 
     \tilde{u},\, D(\theta \tilde{u})\in \sS^{\infty}_{w,\sF}(0,t;L^{\infty}(\cD))
     ,\quad \forall\,t\in(0,T),
 $$
 then a.s. for {all} $t\in [0,T)$,
  $  \tilde{u}_t= u_t $ a.e. on $\cD$.
\end{enumerate}
\end{thm}

\begin{rmk}
Theorem \ref{main-results} only summarizes the main results given in Section 5 which focus on the well-posedness of BSPDE \eqref{ValueFunction_u} and the resolution of the optimal liquidation problem \eqref{OCP}. In contrast to the other applications of BSPDEs to liquidation problems (see \cite{GraeweHorstQui13,Horst-Qiu-Zhang-14}), the main novelty of Theorem \ref{main-results} lies in the treatments of Neumann boundary condition, the general $q$th-power setting, the refined function space $\cM^1$, and the bounded estimate of gradients.

If all the coefficients $\beta,\sigma,\bar \sigma, \lambda,\eta,\gamma$ are deterministic functions (independent of $\omega\in\Omega$), the optimal control problem is Markovian and the  BSPDE
 \eqref{ValueFunction_u} becomes the parabolic PDE:
\begin{equation*} 
  \left\{\begin{array}{l}
  \begin{aligned}
  -\partial_t u_t(y) &= \bigg[\alpha  D^2u 
+\beta  Du +\lambda  -\frac{|u|^{q^*}}{(q^*-1)|\eta|^{q^*-1}}-\mu(\cZ) u
\\
&\quad \quad
+\int_{\mathcal{Z}}\frac{\gamma_{\cdot}(\cdot,z) u}{(|\gamma_{\cdot}(\cdot,z)|^{q^*-1}+|u |^{q^*-1})^{q-1}}\,\mu(dz)\bigg](t,y),
\quad  (t,y)\in[0,T)\times\cD,\\
Du_t(a)&=0,\quad t\in[0,T).\\
u_T(y)&=\infty, \quad  y\in\bar{\cD}.
    \end{aligned}
  \end{array}\right.
\end{equation*}
We would claim that our results are new even in such Markovian cases.
\end{rmk}

\section{A link between FBSDEs and BSPDEs}

\begin{thm}\label{thm-ito-wentzell}
 Let Assumption \ref{AssHJB} be satisfied and suppose that with probability 1, for each $t\in[0,T]$,
 \begin{align*}
    \Phi_t(y)&=\Phi_T(y)+\int_t^T\! \left[\alpha_t(y)D^2\Phi
+\sigma^*D\Psi+\beta D\Phi+g\right](s,y)\,ds\\
&\quad 
	-\int_t^T\!\Psi_r(y)\,dW_r,\quad  \text{dy-a.e.},
  \end{align*}
with $D\Phi _t(a)=0$, $(\Phi,\Psi)\in \cH^1$, $g\in \sL^2_{\sF}(0,T;H^{1,2}(\cD)) $, and 
   \begin{align}
   D\Phi\in \sS^{\infty}_{w,\sF}(0,T;L^{\infty}(\cO)),
   \quad \text{i.e., }
   \esssup_{\omega\in\Omega} \sup_{t\in[0,T]} \|D\Phi(\omega, t,\cdot)\|_{L^{\infty}(\cD)} <\infty.  \label{eq-bd-dphi}
   \end{align}
   Then, for each $y\in\cD$ and $t\in [0,T]$, it holds almost surely that 
 \begin{align}
      \Phi_t(y)&= \Phi_T(y^{t,y}_T)
     +\!\int_t^T   g_r(y^{t,y}_r)  \,dr   
     -\int_t^T\left( \Psi_r(y^{t,y}_r)    
     		+ D\Phi_r(y^{t,y}_r)\sigma_r(y^{t,y}_r)\right)\,dW_r
     \nonumber\\
     &
     -\int_t^T D\Phi_r(y^{t,y}_r)\bar\sigma_r(y^{t,y}_r)\,dB_r
     .
     \label{eq-ito-kunita}
   \end{align}
\end{thm}

\begin{proof}
W.l.o.g., we only need to prove \eqref{eq-ito-kunita} for $t=0$. Setting
\[
H_t(y)=\left[\alpha D^2\Phi
+\sigma^*D\Psi+\beta D\Phi+g\right](t,y),
\]
one has $H\in \sL^2(0,T;L^2(\cD))$. 

The theory of Sobolev spaces allows us to  extend $H^{k,2}(\cD)$ to $H^{k,2}(\bR)$ for integers $k\geq 1$. In particular, when $k=1,2$, the bounded linear extension operator can be constructed (as in \cite[Pages 254-257]{Evans-1998-PDE}) as follows: for each $\zeta\in H^{1,2}(\cD)$ or $\zeta\in H^{2,2}(\cD)$,
\begin{equation*}
\mathcal E \zeta(y)\triangleq
  \left\{\begin{array}{l}
  \begin{aligned}
  &\zeta(y), \quad &&\text{if } y\in [a,\infty);  \\
        & -3\zeta(-y)+4\zeta(-y/2) ,\quad &&\text{if }y\in (-\infty,a].
    \end{aligned}
  \end{array}\right.
\end{equation*}
 $\mathcal E \zeta$ is called an extension of $\zeta$ to $\bR$. 
   Then it is easy to check that with probability 1, for all $t\in[0,T]$,
 \begin{align}
    \mathcal E \Phi_t(y)=\mathcal E\Phi_T(y)+\int_t^T\!  \mathcal E H_r(y)\,dr-\int_t^T\mathcal E \Psi_r(y)\,dW_r,\quad \textrm{for dy-a.e. } y\in\bR. \label{eq-lem-ito-kunita-1}
  \end{align}
  
For each $y\in \cD$, applying the generalized It\^o-Kunita-Wentzell formula of Corollary \ref{Ito-Wentzell-cor} to equation \eqref{eq-lem-ito-kunita-1} yields that with probability 1,
 \begin{align}
   &   \mathcal E \Phi_0(x+y^{0,y}_0)
      -\mathcal E\Phi_T(x+y^{0,y}_T)\nonumber \\
      &
      =\int_0^T\bigg[
      \mathcal E H _t(x+y^{0,y}_t)
      - \alpha(t,y^{0,y}_t)D^2\mathcal E\Phi _t(x+y^{0,y}_t)
      -\sigma^*_t(y^{0,y}_t) D\mathcal E \Psi_t(x+y^{0,y}_t)   \notag \\
&\quad\quad      -\beta_t(y^{0,y}) D\mathcal E\Phi _t(x+y^{0,y}_t)   \bigg]\,dt
-\int_0^T D\mathcal E \Phi_t(x+y^{0,y}_t) \,dL_t
 - \int_0^TD\mathcal E \Phi _t(x+y^{0,y}_t) \bar\sigma_t(y^{0,y}_t)\,dB_t
      \nonumber\\
   &\quad
      -\int_0^T \left[ \mathcal E \Psi_t(x+y^{0,y}_t)+ D\mathcal E \Phi _t(x+y^{0,y}_t)\sigma_t(y^{0,y}_t)\right]\,dW_t      , \quad \text{for dx-a.e. }x\in\bR. \nonumber
 \end{align}
 Notice that for all $x\geq 0$, $y_t^{0,y}+x \geq a$ a.s. for all $t\in[0,T]$. This, together with the definition of function $H$, implies further that with probability 1,
  \begin{align}
   &    \Phi_0(x+y^{0,y}_0)
      -\Phi_T(x+y^{0,y}_T)\nonumber \\
      &
      =I_1(x)+I_2(x)+I_3(x)+\int_0^T\!\!\!\!   g_t(x+y^{0,y}_t)\,dt
      -\int_0^T\!\!\!\! D \Phi_t(x+a) \,dL_t
   - \int_0^T\!\!\!\!  D  \Phi _t(x+y^{0,y}_t) \bar\sigma_t(y^{0,y}_t)\,dB_t
       \nonumber\\
                &\quad
   -\int_0^T\left[  \Psi_t(x+y^{0,y}_t)+ D  \Phi_t(x+y^{0,y}_t) \sigma_t(y^{0,y}_t) \right]\,dW_t
      , \quad \text{for dx-a.e. }x\in[0,\infty), \label{eq-ito-kunita-1}
 \end{align}
 where we note that $dL_t=\textbf{1}_{\{y^{0,y}_t=a\}}\,dL_t$, and
 \begin{align*}
 I_1(x)&=\int_0^T  \left[\alpha(t,x+y^{0,y}_t)- \alpha(t,y^{0,y}_t) \right]D^2  \Phi _t(x+y^{0,y}_t)\, dt,
 		\\
I_2(x)&= \int_0^T       \left[ \sigma^*_t(x+y^{0,y}_t) -\sigma^*_t(y^{0,y}_t) \right] D  \Psi_t(x+y^{0,y}_t) \,dt,
		\\
I_3(x)&=  \int_0^T      \left[ \beta_t(x+y^{0,y}_t) -\beta_t(y^{0,y}_t) \right] D  \Phi_t(x+y^{0,y}_t) \,dt.
 \end{align*}
 
 In particular, for each $\delta>0$ it holds that with probability 1,
   \begin{align*}
   &    \frac{1}{\delta}\int_0^{\delta}\left(\Phi_0(x+y^{0,y}_0)
      -\Phi_T(x+y^{0,y}_T) \right) \,dx\nonumber \\
      &
      =\frac{1}{\delta}\int_0^{\delta} \left(I_1+I_2+I_3\right)(x)\,dx
      +\int_0^T \frac{1}{\delta}\int_0^{\delta} g_t(x+y^{0,y}_t)\,dx\,dt
      -\int_0^T \frac{1}{\delta}\int_0^{\delta}  D \Phi_t(x+a) \,dx\,dL_t
     \\
     &\quad
      -\int_0^T \frac{1}{\delta}\int_0^{\delta}   \left[  \Psi_t(x+y^{0,y}_t)+ D  \Phi_t(x+y^{0,y}_t) \sigma_t(y^{0,y}_t) \right]         \,dx\,dW_t
         - \int_0^T \frac{1}{\delta}\int_0^{\delta} D  \Phi _t(x+y^{0,y}_t) \bar\sigma_t(y^{0,y}_t)\,dx\,dB_t,
 \end{align*}
 where we have used the Fubini's theorem. By Sobolev's embedding theorem,  the space $H^{m,2}(\cD)$ is continuously embedded into continuous function space $C^{m-1,\frac{1}{4}}(\cD)$ for $m\geq 1$. Thus, to obtain the desired relation \eqref{eq-ito-kunita}, we may let $\delta\rightarrow 0^+$. Notice that with probability 1,
 \begin{align*}
&\left|  \frac{1}{\delta}\int_0^{\delta} I_2(x)\,dx   \right|   \\
&=\left| \int_0^T   \frac{1}{\delta}\int_0^{\delta}    \left[ \sigma^*_t(x+y^{0,y}_t) -\sigma^*_t(y^{0,y}_t) \right] D  \Psi_t(x+y^{0,y}_t) \, dx dt   \right|
	\\
&\leq
 \int_0^T\int_0^{\delta}    \frac{\Lambda x}{\delta}   \left| D\Psi_t(x+y^{0,y}_t) \right| \,dxdt
 	\\
&\leq
 \int_0^T \left(\int_0^{\delta}    \frac{\Lambda^2 x^2}{\delta^2}      dx\right)^{1/2}   \left(\int_0^{\delta}\left| D\Psi_t(x+y^{0,y}_t) \right|^2 \,dx\right)^{1/2} dt
 	\\
 &\leq
 \frac{ \delta^{1/2}\Lambda }{\sqrt{3}} \cdot \sqrt{T}    \left(  \int_0^T\int_a^{\infty} \left| D\Psi_t(x) \right|^2 \,dxdt \right)^{1/2}
 	\\
&\rightarrow 0,\quad \text{as }\delta \rightarrow 0^+,
 \end{align*}
 and we have the similar calculations for $I_1$ and $I_3$. It remains to check the terms of stochastic integrals as all the other terms follow straightforwardly from the dominated convergence theorem. Indeed, the Ito isometry and the embedding theorems yield
 \begin{align*}
&E\bigg[ \Big|\int_0^T \frac{1}{\delta}\int_0^{\delta} D  \Phi _t(x+y^{0,y}_t) \bar\sigma_t(y^{0,y}_t) \,dx\,dB_t 
- \int_0^T  D  \Phi _t(y^{0,y}_t) \bar\sigma _t(y^{0,y}_t) \,dB_t\Big|^2\bigg]
\\
&=
E\bigg[ \int_0^T \Big| \frac{1}{\delta}\int_0^{\delta} \left( D  \Phi _t(x+y^{0,y}_t) \bar\sigma_t(y^{0,y}_t) 
											 -   D  \Phi _t(y^{0,y}_t) \bar\sigma _t(y^{0,y}_t)
											 		   \right)\,dx\Big|^2\,dt 
\bigg]
\\
&\leq
E\bigg[ \int_0^T \frac{1}{\delta}\int_0^{\delta} \left| D  \Phi _t(x+y^{0,y}_t) \bar\sigma_t(y^{0,y}_t) 
											 -D  \Phi _t(y^{0,y}_t) \bar\sigma _t(y^{0,y}_t)
											 			\right|^2\,dx\,dt 
\bigg]
\\
&\leq 
C \| D \Phi \|_{\sL^2_{\sF}(0,T; H^{1,2}(\cD))} ^2 \frac{1}{\delta}\int_0^{\delta}|x|^{1/2} dx
\\
&\leq
C \| D  \Phi \|_{\sL^2_{\sF}(0,T; H^{1,2}(\cD))} ^2 \cdot \frac{2\delta^{1/2}}{3}
\\
&\rightarrow 0, \quad \text{as }\delta \rightarrow 0^+,
 \end{align*}
 and the other term of the stochastic integral follows in a similar way. This completes the proof.
\end{proof}

\begin{rmk}\label{rmk-FBSDE-BSPDE}
The gradient boundedness \eqref{eq-bd-dphi} plays an important role when we use the dominated convergence theorem to reach the limit:
\begin{align}
\lim_{\delta\rightarrow 0^+} 
\int_0^T \frac{1}{\delta}\int_0^{\delta}  D \Phi_t(x+a) \,dx\,dL_t
=\int_0^T D \Phi_t(a)\,dL_t
=0.\label{eq-lim-L}
\end{align}
Otherwise, $\int_0^T \frac{1}{\delta}\int_0^{\delta}  D \Phi_t(x+a) \,dx\,dL_t$ is not dominated, and as the regularity 
$$\Phi \in \sS^2_{\sF}(0,T;H^{1,2}(\cD)) \cap\sL^2_{\sF}(0,T;H^{2,2}(\cD)) $$
 only makes sense of $\mathcal K (x):=\int_0^T  D \Phi_t(x+a) \,dL_t$ with $\mathcal K \in L^1(\Omega; L^2(\cD))$, it would be difficult to identify the limit \eqref{eq-lim-L} or make sense of \eqref{eq-ito-kunita} at the specific point $x=0$.
\end{rmk}

\section{ Neumann problem of semilinear BSPDEs and a comparison theorem}

 In this section, we shall prove the existence and uniqueness of solutions in space $\cM^1$ for Neumann problems of general semilinear BSPDEs, and a comparison theorem is also established under the boundedness assumption on gradients of solutions. 
 
 Consider the following Neumann problem of the semilinear BSPDE  with Lipschitz continuous coefficients and finite terminal value:
\begin{equation}\label{Lip_bspde}
\left\{
  \begin{split}
    -dv_t(y)&=\bigl[\alpha D^2v+\sigma^* D\zeta+F(t,y,Dv,v,\zeta)\bigr]\,dt -\zeta_t(y)\,dW_t,  \quad   (t,y)\in(0,T)\times\cD,\\
    Dv_t(a)&=0,  \quad t\in[0,T],\\
    v_T(y)&=G(y), \quad \forall y\in\bar{\cD}.
  \end{split}
  \right.
\end{equation}

\begin{ass}\label{AssLip}
\begin{enumerate}[(1)]
\item
 (Lipschitz continuity) $F: \Omega \times [0,T]\times \cD \times \bR \times \bR \times \bR^d\rrow \bR$ is $\sP\times \sB(\cD)\times \sB(\bR)\times \sB(\bR)\times \sB(\bR^d)\rrow\sB(\bR)$ measurable, $F_0(t,y):=F(t,y,0,0,0)\in \sM^2_{\sF}(0,T;L^2(\cD))$, and there is a constant $K>0$ such that for any $(p,q,r),(\bar{p},\bar{q},\bar{r})\in\bR\times\bR\times\bR^d$, it holds that
$$|F(t,y,p,q,r)-F(t,y,\bar{p},\bar{q},\bar{r})|\leq K (|p-\bar{p}|+|q-\bar{q}|+|r-\bar{r}|),\quad (t,y)\in[0,T]\times \cD.$$
\item $G\in L^{\infty}(\Omega,\sF_T;H^1(\cD))$ with $DG(a):=\lim_{\delta\rightarrow 0^+} \frac{1}{\delta}\int_a^{a+\delta}DG(x)\, dx=0$ a.s.
\end{enumerate}
\end{ass}

\begin{defn}
A pair of processes $(v,\zeta)$ is a strong solution to BSPDE \eqref{Lip_bspde} if  $(v,\zeta)\in\cH^1$ satisfying a.s. for all $t\in[0,T]$,
$$v_t(y)=G(y)+\int_t^T\bigl[\alpha D^2v+\sigma^* D\zeta+F(s,y,Dv,v,\zeta)\bigr](s,y)\,ds -\int_t^T\zeta_t(y)\,dW_s,  \quad \text{for dy-a.e. }y\in \cD,$$
and $Dv_t(a)=0$.
\end{defn}

Under the above assumptions, we can prove the following a priori estimate for the strong solutions of BSPDE \eqref{Lip_bspde}.
\begin{prop}\label{priori_est}
Let Assumptions \ref{AssHJB} and \ref{AssLip} hold. If $(v,\zeta)\in\cH^1$ is a strong solution to BSPDE \eqref{Lip_bspde}, then the strong solution is unique, and we have $(v,\,\zeta)\in\cM^1$ with 
$$\|(v,\zeta)\|_{\cM^1}\leq C\Big(\|G\|_{L^{\infty}(\Omega; H^1(\cD))}+\|F_0\|_{\sM^2_{\sF}(0,T;L^2(\cD))}\Big),$$
where the constant $C$ depends only on $\kappa, K,T$ and $\Lambda$.
\end{prop}
\begin{proof}
\textbf{Step 1:}
Applying the generalized It\^o's Lemma for square norms gives
\begin{equation}\label{inequ_1}
\begin{split}
\|v_t\|^2+\int_t^T\|\zeta_s\|^2\,ds=&\|G\|^2+2\int_t^T\langle v_s, \alpha_s D^2v_s\rangle\,ds+2\int_t^T\langle v_s,\sigma^*_s D\zeta_s\rangle\,ds\\+&2\int_t^T\langle v_s,F(s,y,Dv,v,\zeta)\rangle\,ds-2\int_t^T\langle v_s,\zeta_sdW_s\rangle.
\end{split}
\end{equation}
In view of the zero Neumann boundary condition, we have
\begin{equation*}
\begin{split}
2\int_t^T\langle v_s, \alpha_s D^2v_s\rangle\,ds&=-2\int_t^T\langle \alpha_s Dv_s,Dv_s\rangle\,ds-2\int_t^T\langle D\alpha_s  v_s,Dv_s\rangle\,ds\\
\leq &-2\kappa \int_t^T\|Dv_s\|^2\,ds+\frac{\Lambda^2}{\epsilon_1}\int_t^T\|v_s\|^2\,ds+\epsilon_1\int_t^T\|Dv_s\|^2\,ds.
\end{split}
\end{equation*}
Further,
$$2\int_t^T\langle v_s,\sigma_s^*D\zeta_s\rangle\,ds\leq \frac{\Lambda^2}{\epsilon_2}\int_t^T\|v_s\|^2\,ds+\epsilon_2\int_t^T\|D\zeta_s\|^2\,ds \quad \text{with }\epsilon_2\in (0,1],$$
and 
\begin{equation*}
\begin{split}
2\int_t^T\langle v_s,F(t,\cdot,Dv,v,\zeta)\rangle\,ds&\leq \Big[1+K^2\left(1+\frac{1}{\epsilon_3}+\frac{1}{\epsilon_4}\right)\Big]\int_t^T\|v_s\|^2\,ds+\epsilon_3\int_t^T\|Dv_s\|^2\,ds\\
&+\epsilon_4\int_t^T\|\zeta_s\|^2\,ds+\int_t^T\|F_0(s,\cdot)\|^2\,ds. 
\end{split}
\end{equation*}
Taking conditional expectations on both sides of \eqref{inequ_1} and choosing appropriate $\epsilon_1, \epsilon_3$ and $\epsilon_4$ such that $2\kappa-\epsilon_1-\epsilon_3\geq\kappa$ and $\epsilon_4\leq \frac{1}{2}$, we have
\begin{equation}
\begin{split}
&\|v_t\|^2+E\left[\int_t^T\!\!\|Dv_s\|^2ds+\int_t^T\!\!\|\zeta_s\|^2ds \Big| \sF_t    \right]\\
\leq C& E  \left[\|G\|^2+\int_t^T\!\!\|F_0(s,\cdot)\|^2ds+ \frac{\Lambda^2}{\epsilon_2}\int_t^T\|v_s\|^2\,ds+\epsilon_2\int_t^T\!\!\|D\zeta_s\|^2ds  \Big| \sF_t   \right], \label{est-t-e}
\end{split}
\end{equation}
where we have used the fact that the Burkholder-Davis-Gundy inequality yields
\begin{align*}
E\left[\sup_{\tau\in[t,T]}\left|\int_{\tau}^T\langle v_s,\zeta_sdW_s\rangle\right|\right] 
&\leq C E\left(\int_t^T\|v_s\|^2\|\zeta_s\|^2ds\right)^{\frac{1}{2}}
\\
&\leq C E\left[\sup_{\tau\in[t,T]}\|v_{\tau}\|^2+\int_t^T\|\zeta_s\|^2\,ds\right] <\infty,
\end{align*}
and thus,
$$
E\left[
2\int_t^T\langle v_s,\zeta_sdW_s\rangle   \Big|\sF_t
\right]=0.
$$

Taking the supremum with respect to the time variable on both sides of \eqref{inequ_1} and choosing appropriate $\epsilon_i (i=1,3,4)$, we arrive at 
\begin{equation}\label{apriori-step1}
\begin{split}
&\sup_{s\in[t,T]}\|v_s\|^2+ \sup_{\tau\in[t,T]} E\left[ \int_{\tau}^T[\|v_s\|^2+\|Dv_s\|^2]ds+\int_{\tau}^T\|\zeta_s\|^2ds \Big| \sF_{\tau}\right]\\
\leq C& \sup_{\tau\in [t,T]}E\left[\|G\|^2+\int_{\tau}^T\!\!\|F_0(s,\cdot)\|^2ds +  \frac{\Lambda^2}{\epsilon_2}\int_{\tau}^T\|v_s\|^2\,ds +\epsilon_2\int_{\tau}^T\!\!\|D\zeta_s\|^2ds 
			\Big| \sF_{\tau}    \right], \quad \text{a.s.},
\end{split}
\end{equation}
with the constant $C$ independent of $\epsilon_2$.

\noindent \textbf{Step 2:}   Take the derivative in $y$ on both sides of BSPDE \eqref{Lip_bspde} and write $v^{\prime}_t(y):=Dv_t(y)$ and $\zeta^{\prime}_t(y):=D\zeta_t(y)$. Then the pair $(v^{\prime},\zeta^{\prime})$ satisfies the following BSPDE with the zero Dirichlet boundary condition,
 \begin{equation*}
\left\{
  \begin{split}
    -dv^{\prime}_t(y)&=\bigl[D(\alpha Dv^{\prime}) +D(\sigma^*\zeta^{\prime})+D\left( F(t,y,v^{\prime},v,\zeta)\right) \bigr](t,y)\,dt -\zeta^{\prime}_t(y)\,dW_t,\\
   v^{\prime}_t(a)&=0,\\
    v^{\prime}_T(y)&=DG(y), \quad \forall y\in\bar{\cD}.
  \end{split}
  \right.
\end{equation*}


Applying the generalized It\^o's Lemma to the square norm and the integration-by-parts formula, we obtain
\begin{equation*}
\begin{split}
&\|v^{\prime}_t\|^2+\int_t^T\|\zeta^{\prime}_s\|^2\,ds + 2\int_t^T\langle v'_s,\zeta'_sdW_s\rangle-\|DG\|^2 \\
&=-2\int_t^T\langle D v^{\prime}_s, \alpha_s Dv^{\prime}_s\rangle\,ds-2\int_t^T\langle Dv^{\prime}_s,\sigma^*_s \zeta^{\prime}_s\rangle\,ds
	-2\int_t^T\langle Dv^{\prime}_s,{F}(t,y, v^{\prime},  v,\zeta)\rangle\,ds
		\\
&
\leq
-2\int_t^T\langle D v^{\prime}_s, \alpha_s Dv^{\prime}_s\rangle\,ds+(1+\epsilon_5)\int_t^T \|\sigma_s Dv_s^{\prime} \|^2 \,ds+  \frac{1}{1+\epsilon_5}\int_t^T\|  \zeta^{\prime}_s\|^2\,ds
\\
&\quad
	+ \epsilon_6\int_t^T \|Dv^{\prime}_s\|^2 \,ds + C(\epsilon_6,K) \int_t^T \left(\|F_0(s,\cdot)\|^2 + \|v^{\prime}_s\|^2 + \|v_s\|^2 + \|\zeta_s\|^2\right) \,ds,
\end{split}
\end{equation*}
which together with the estimate
\begin{align*}
2E\left[\sup_{\tau\in[t,T]}\left|\int_{\tau}^T\langle v^{\prime}_s,\zeta^{\prime}_sdW_s\rangle\right|\right] 
&\leq C E\left(\int_t^T\|v^{\prime}_s\|^2\|\zeta^{\prime}_s\|^2ds\right)^{\frac{1}{2}}
\\
&\leq  E\sup_{\tau\in[t,T]}\|v'_{\tau}\|^2+{C}E\int_t^T\|\zeta'_s\|^2\,ds   
	\\
&< \infty,
\end{align*}
implies that (by taking conditional expectations of both sides)
\begin{align*}
& \|v^{\prime}_\tau\|^2+  
	E\left[  \int_{\tau}^T\|\zeta^{\prime}_s\|^2\,ds -\|DG\|^2 \Big|\sF_\tau \right] \\
&\leq 
-2 
E\bigg[
	 \int_{\tau}^T\langle D v^{\prime}_s, \alpha_s Dv^{\prime}_s\rangle\,ds \Big|\sF_\tau\bigg]
	+E\bigg[ (1+\epsilon_5) \int_{\tau}^T \|\sigma_s Dv_s^{\prime} \|^2 \,ds+ 
 	\frac{1}{1+\epsilon_5} \int_{\tau}^T\|  \zeta^{\prime}_s\|^2\,ds
\\
	&\quad
	+ \epsilon_6 \int_{\tau}^T \|Dv^{\prime}_s\|^2 \,ds 
		+ C(\epsilon_6,K) \int_{\tau}^T \left(\|F_0(s,\cdot)\|^2 + \|v^{\prime}_s\|^2 + \|v_s\|^2 + \|\zeta_s\|^2\right) \,ds
	\Big|  {\sF_{\tau}}\bigg]
\\
&\text{(by relation \eqref{apriori-step1} and Assumption \ref{AssHJB})}
\\
&\leq 
-(\kappa-\eps_5 \Lambda^2) E \left[ \int_\tau^T\| D v^{\prime}_s \|^2 \,ds  \Big| \sF_{\tau}\right] 
	+  \frac{1}{1+\epsilon_5} E \left[ \int_\tau^T\|  \zeta^{\prime}_s\|^2\,ds \Big| \sF_\tau \right]
	+ \epsilon_6 E\left[ \int_\tau^T \|Dv^{\prime}_s\|^2 \,ds \Big| \sF_\tau\right] \\
&\quad
+ C(\epsilon_6,K,\Lambda,\kappa)
	E\left[\|G\|^2+\int_\tau^T \left( \|F_0(s,\cdot)\|^2 +  \frac{\Lambda^2}{\epsilon_2}\|v_s\|^2+\epsilon_2 \|\zeta^{\prime}_s\|^2 \right) \, ds \Big|\sF_\tau  \right], \quad \text{a.s. for }\tau\in[0,T].
\end{align*}
Letting 
$$ \epsilon_5=\frac{\kappa}{4\Lambda^2}, \quad \epsilon_6=\frac{\kappa}{4}, \quad\text{and }
\epsilon_2 < \frac{\epsilon_5}{2(1+\epsilon_5) C(\epsilon_6,K,\Lambda,\kappa)},
$$
 and taking the supremum with respect to the time variable, we have 
\begin{align*}
&\sup_{s\in[t,T]}\|v^{\prime}_s\|^2+ \sup_{\tau\in[t,T]} E\left[  \int_\tau^T\left( \|D v^{\prime}_s\|^2 + \|\zeta^{\prime}_s\|^2\right)\,ds  \Big|\sF_\tau \right] 
\\
&\leq C \sup_{\tau\in[t,T]}E \left[ \|G\|_{H^{1,2}(\cD)}^2 
+\int_\tau^T \left( \|F_0(s,\cdot)\|^2 +  \|v_s\|^2 \right) \, ds
\Big| \sF_\tau \right],\text{ a.s. for all }t\in[0,T].
\end{align*}
This together with relation \eqref{apriori-step1} finally implies $(v,\zeta)\in \cM^1$ with the desired estimate:
\begin{align}
&\esssup_{\omega\in\Omega} \sup_{s\in[0,T]}\|v_s\|_{H^{1}(\cD)}^2
+\esssup_{\omega\in\Omega} \sup_{\tau\in[0,T]} E\left[ \int_\tau^T\left( \|v_s\|_{H^{2}(\cD)}^2  + \|\zeta_s\|_{H^{1}(\cD)}^2\right)\,ds \Big| \sF_\tau  \right] 
\nonumber\\
&
\leq C\esssup_{\omega\in\Omega}  \|G\|_{H^{1}(\cD)}^2 
+
\esssup_{\omega\in\Omega} \sup_{\tau\in[0,T]} E \left[ \int_\tau^T \|F_0(s,\cdot)\|^2  \, ds\Big|\sF_\tau
\right]. \label{est-apriori-Q}
\end{align}

\noindent \textbf{Step 3:} 
The uniqueness follows as a consequence of the associated estimate with similar arguments. Indeed, for any two strong solutions $(\hat v,\, \hat \zeta)$ and $(\bar v,\,\bar \zeta)$, we may apply the Lipschitz continuity \ref{AssLip} (1) and the computations in Steps 1 and 2 to the difference $(\hat v-\bar v,\,\hat \zeta-\bar \zeta)$, and reach a similar estimate \eqref{est-apriori-Q} with $(G,\,F_0)$ replaced by $(0,0)$, which then yields the uniqueness.
 \end{proof}


\begin{thm}\label{exi_uni_Lip}
Let Assumptions \ref{AssHJB} and \ref{AssLip} hold. The BSPDE \eqref{Lip_bspde} admits a unique strong solution $(v,\zeta)$ in $\cM^1$, and 
$$\|(v,\zeta)\|_{\cM^1}\leq C\Big(\|G\|_{L^{\infty}(\Omega; H^1(\cD))}+\|F_0\|_{\sM^2_{\sF}(0,T;L^2(\cD))}\Big),$$
with the constant $C$ depending on $\kappa, \Lambda, T$ and $K$.
\end{thm}

\begin{proof}
\textbf{Step 1:}We first show that the model equation
\begin{equation}\label{simpleBspde}
\left\{
  \begin{split}
    -dv_t(y)&=\bigl[D^2v_t(y)+h_t(y)\bigr]\,dt -\zeta_t(y)\,dW_t,  \quad \forall (t,y)\in(0,T)\times\cD,\\
    Dv_t(a)&=0,  \quad t\in[0,T],\\
    v_T(y)&=G(y), \quad \forall y\in\bar{\cD}.
  \end{split}
  \right.
\end{equation}
with $h\in  \sM^2_{\sF}(0,T;L^2(\cD))$ admits a unique strong solution $(v,\zeta)\in\cM^1$. In fact, the uniqueness follows directly from   Proposition \ref{priori_est}. It remains to prove the existence. 

Find a sequence $(G^n,h^n)\in L^{\infty}(\Omega,\sF_T;H^2(\cD))\times\sM^2_{\sF}(0,T;H^1(\cD))$,  such that $(G^n,h^n)\rrow(G,h)$ in $L^\infty (\Omega,\sF_T;H^1(\cD))\times\sM^2_{\sF}(0,T;L^2(\cD))$ with $DG^n(a)=0$.  Consider the following PDE with the Neumann boundary condition
\begin{equation}\label{pde}
\left\{
  \begin{split}
    -\frac{\partial}{\partial t}\tilde{v}^n_t(y)&=\bigl[D^2\tilde{v}^n_t(y)+h^n_t(y)\bigr], \quad \forall y\in\cD,\\
    D \tilde{v}^n_t(a)&=0,  \quad t\in[0,T],\\
   \tilde{v}^n_T(y)&=G^n(y), \quad \forall y\in\bar{\cD}.
  \end{split}
  \right.
\end{equation}
By \cite[Proposition 7.18]{Lieberman}, it has a unique solution $\tilde{v}$, such that $\tilde{v}$, $D\tilde{v}$, $D^2\tilde{v}$ and $\frac{\partial }{\partial t}\tilde{v}$ are lying in  $ L^2(\Omega,\sF_T;L^2(0,T;L^2(\cD)))$. Define $v_t^n(y)=E[\tilde{v}^n_t(y)|\sF_t]$ and solve for $\zeta^n$ the backward SDE
\begin{equation}\label{bsde}
v^n_t(y)=G^n(y)+\int_t^T\bigl[D^2v^n_s(y)+h^n_s(y)\bigr]\,ds -\int_t^T\zeta^n_s(y)\,dW_s,  \quad \forall y\in\cD,
\end{equation}
Then, $v^n\in \cS^2_{\sF}(0,T;H^1(\cD))\cap\sL^2_{\sF}(0,T;H^2(\cD))$  and $\zeta^n\in\sL^2_{\sF}(0,T;L^2(\cD))$. 

In order to prove that $(v^n,\zeta^n)$ is the strong solution to \eqref{simpleBspde} with $(G^n,h^n)$, we claim that $\zeta^n\in\sL^2_{\sF}(0,T;H^1(\cD))$. In fact, taking the derivative in $y$ of both sides of \eqref{pde}, $D\tilde{v}^n$ is a strong solution of the following PDE with the Dirichlet boundary condition
\begin{equation*}
\left\{
  \begin{split}
    -\frac{\partial }{\partial t}D\tilde{v}^n&=\bigl[D^2(D\tilde{v}^n_t(y))+Dh^n_t(y)\bigr], \quad \forall y\in\cD,\\
   D\tilde{v}^n(t,a)&=0,  \\
   D\tilde{v}^n(T,y)&=DG^n(y), \quad \forall y\in {\cD}.
  \end{split}
  \right.
\end{equation*}
Similarly, it is easy to prove that $(Dv^n,D\zeta^n)$ is solution of the backward SDE \eqref{bsde} associated to $(DG^n,Dh^n)$. Then, $Dv^n\in \cS^2_{\sF}(0,T;H^1(\cD))\cap\sL^2_{\sF}(0,T;H^2(\cD))$  and $D\zeta^n\in\sL^2_{\sF}(0,T;L^2(\cD))$. Therefore, $(v^n,\zeta^n)\in\cH^1$ is strong solution of \eqref{simpleBspde} associated with $(G^n,h^n)$, and by Proposition \ref{priori_est}, we further have $(v^n,\zeta^n)\in\cM^1$. Moreover, the a priori estimate in Proposition \ref{priori_est} yields the convergence of $(v^n,\zeta^n)$ and then the existence of the strong solution $(v,\zeta)\in \cM^1$ for BSPDE \eqref{simpleBspde}. 

 \noindent \textbf{Step 2:}  We prove the unique existence of the strong solution to \eqref{Lip_bspde} with the method of continuation.  For each $\eps\in[0,1]$, define 
 $$Q_{ \varepsilon}(v,\zeta)=(1- \varepsilon)D^2v+ \varepsilon[\alpha D^2v+\sigma^* D\zeta+F(t,y,Dv,v,\zeta)], \text{ for }(v,\zeta)\in\cM^1.$$ 
 For each $(\hat v,\hat\zeta)\in \cM^1$ and $\eps_0,\eps\in[0,1]$, consider BSPDE
 \begin{equation}\label{Lip_bspde-eps}
\left\{
  \begin{split}
    -dv_t(y)&=\bigl[   Q_{ \varepsilon_0}(v,\zeta) + Q_{ \varepsilon}(\hat v,\hat \zeta)- Q_{ \varepsilon_0}(\hat v,\hat \zeta)     \bigr](t,y)\,dt -\zeta_t(y)\,dW_t,\\
    Dv_t(a)&=0,  \quad t\in[0,T],\\
    v_T(y)&=G(y).
  \end{split}
  \right.
\end{equation}
Notice that for each $(\hat v,\hat\zeta)\in \cM^1$, we have 
\begin{align*}
Q_{ \varepsilon}(\hat v,\hat \zeta)- Q_{ \varepsilon_0}(\hat v,\hat \zeta) 
	=
	( \varepsilon -\eps_0) \left(- D^2\hat v+\alpha D^2\hat v+\sigma^* D\hat\zeta+F(t,y,D\hat v,\hat v,\hat \zeta)\right),
\end{align*}
and that the a priori estimate in Proposition \ref{priori_est} is applicable to the BSPDE \eqref{Lip_bspde-eps} for any $\eps_0,\eps\in [0,1]$. In particular, when $\eps_0=0$, it has been shown in Step 1 that  the BSPDE \eqref{Lip_bspde-eps}  admits a unique strong solution $(v,\zeta)\in\cM^1$ for any $\eps\in[0,1]$ and $(\hat v,\hat\zeta)\in \cM^1$.

Suppose that for some $\eps_0\in [0,1]$, the BSPDE \eqref{Lip_bspde-eps} has a unique strong solution $(v,\zeta)\in\cM^1$ for each given $(\hat v,\hat\zeta)\in \cM^1$. We may define the solution mapping:
$$
\Pi_{\eps}: \, \cM^1 \rightarrow \cM^1, \quad (\hat v,\hat\zeta)\mapsto (v,\zeta).
$$
Then for any $(\hat v^i,\hat \zeta^i)\in \cM^1$, $i=1,2$, set $(\delta \hat v,\delta \hat \zeta)=(\hat v^1-\hat v^2,\hat \zeta^1-\hat \zeta^2)$, and  we have by Proposition \ref{priori_est}
\begin{align*}
\|\Pi_{\eps}(\hat v^1,\hat\zeta^1) -   \Pi_{\eps}(\hat v^2,\hat\zeta^2)\|_{\cM^1}
&\leq
  C |\eps_0-\eps|\, \left\|
  -D^2 \delta \hat v + \alpha D^2\delta \hat v+\sigma^* D \delta \hat\zeta+F(\cdot,\cdot,D\hat v^1,\hat v^1,\hat \zeta^1) \right.
   \\
   &\quad\quad\quad 
   	 \left. -F(\cdot,\cdot,D\hat v^2,\hat v^2,\hat \zeta^2) 
  \right\|_{\sM_{\sF}^2(0,T;L^2)}
  \\
  &\leq \overline C|\eps-\eps_0|\, \left\| (\delta \hat v, \,\delta \hat \zeta)   \right\|_{\cM^1},
\end{align*}
where the constant $\overline C$ is independent of  the pair $(\eps,\eps_0)$. Hence, the mapping $\Pi_{\eps}$ from $\cM^1$ to itself is clearly a contraction mapping and has a unique fixed point  whenever $|\eps_0-\eps|< \delta_0:= \frac{1}{\overline C}$. Starting from $\eps_0=0$, we may arrive at $\eps=1$ in finite steps, and we conclude that $\Pi_1$ has a unique fixed point $(v,\zeta)\in \cM^1$ which is indeed a strong solution of BSPDE \eqref{Lip_bspde}. The uniqueness and estimation follow directly from Proposition \ref{priori_est}.
\end{proof}

We also have the following comparison theorem for the solutions.
\begin{thm}\label{comparison-thm}
Let Assumption \ref{AssHJB} hold. Suppose that $(v^1,\zeta^1)$ and $(v^2,\zeta^2)$ are strong solutions of \eqref{Lip_bspde} with, respectively, $(G^1,F^1)$ and $(G^2,F^2)$ satisfying Assumption \ref{AssLip}. Moreover, we assume both $v^1$ and $v^2$ satisfy the boundedness condition \eqref{eq-bd-dphi}. If $G^1 \leq G^2$ and $F^1(t,y,Dv^1,v^1,\zeta^1)\leq F^2(t,y,Dv^1,v^1,\zeta^1)$, then for each $(t,y)\in [0,T]\times \cD$, it holds that
$$v^1_t(y)\leq v^2_t(y), \quad\text{a.s.} $$
\end{thm}

\begin{proof}
Fix some given $(t,y)\in[0,T]\times\cD$. By Theorem \ref{thm-ito-wentzell}, we have for $i=1,2$
\begin{align}
v^i_t(y)
&=G^i(y_T^{t,y})+\int_t^T\left[ F^i(s,y^{t,y}_s,Dv^i_s(y^{t,y}_s),v^i_s(y^{t,y}_s),\zeta^i_s(y^{t,y}_s))-\beta_s(y^{t,y}_s) Dv^i_s(y^{t,y}_s)\right]\,ds
\nonumber\\
&\quad
- \int_t^T \left(\zeta^i_r(y^{t,y}_r) +Dv^i_r(y^{t,y}_r)\sigma_r(y^{t,y}_r)\right)\,dW_r
     -\int_t^T Dv^i_r(y^{t,y}_r)\bar\sigma_r(y^{t,y}_r)\,dB_r, \quad \text{a.s.}   \label{BSDE-v}
\end{align}
For each $s\in[t,T]$ and $i=1,2$, set
\begin{align*}
Y^i_s
=v^i_s(y^{t,y}_s),\quad
Z^i_s
= (\zeta^i_s(y^{t,y}_s) +Dv^i_s(y^{t,y}_s)\sigma_s(y^{t,y}_s)
\quad \text{and}\quad
\bar Z^i_s
=
Dv^i_s(y^{t,y}_s)\bar\sigma_s(y^{t,y}_s),
\end{align*}
and it follows that
$$
Dv^i_s(y^{t,y}_s) = \frac{\bar\sigma^*_s(y^{t,y}_s) \bar Z^i_s}{|\bar\sigma_s(y^{t,y}_s)|^2} 
\quad \text{and}\quad
\zeta^i_s(y^{t,y}_s) = Z^i_s-   \frac{\bar\sigma^*_s(y^{t,y}_s) \bar Z^i_s}{|\bar\sigma_s(y^{t,y}_s)|^2}  \sigma_s(y^{t,y}_s) .
$$
Then the equation \eqref{BSDE-v} may be written equivalently as the following BSDE:
\begin{align}
&Y^i_t-G^i(y_T^{t,y})\nonumber\\
&=
\int_t^T\left[ F^i\left(s,y^{t,y}_s,  \frac{\bar\sigma^*_s(y^{t,y}_s) \bar Z^i_s}{|\bar\sigma_s(y^{t,y}_s)|^2}  ,Y^i_s,  Z^i_s-   \frac{\bar\sigma^*_s(y^{t,y}_s) \bar Z^i_s}{|\bar\sigma_s(y^{t,y}_s)|^2}  \sigma_s(y^{t,y}_s)\right)-\beta_s(y^{t,y}_s)  \frac{\bar\sigma^*_s(y^{t,y}_s) \bar Z^i_s}{|\bar\sigma_s(y^{t,y}_s)|^2} \right]\,ds
\nonumber\\
&\quad
- \int_t^T Z^i_r\,dW_r
     -\int_t^T\bar Z^i_r\,dB_r, \quad \text{a.s., for }i=1,2.    \label{BSDE-Y}
\end{align}
Under Assumption \ref{AssHJB}, it is easy to check that the coefficients $F^i$ in BSDEs \eqref{BSDE-Y}  are Lipschitz continuous with respect to $(Y^i, Z^i, \bar Z^i)$. Then,  the standard comparison theorem for BSDEs (see \cite[Theorem 2.2]{Karoui_Peng_Quenez} for instance) indicates that $Y_t^1\leq Y^2_t$ a.s., i.e., $v^1_t(y)\leq v_t^2(y)$ a.s.
\end{proof}

\begin{rmk}\label{rmk-comparison-thm}
In view of the above proof, we see that the terminal time $T$ in Theorem \ref{comparison-thm} may be a stopping time. Besides, the above proof is not standard for BSPDEs.  With a standard method, we may first prove the It\^o formula for the square norm of $(v^1-v^2)^+$ and follow a similar way to \cite[Proposition A.2]{Horst-Qiu-Zhang-14} to complete the proof. Nevertheless, for the sake of simplicity, we would not seek such a generality in this paper.
\end{rmk}

\section{ Unique existence of strong solution to BSPDE \eqref{ValueFunction_u} and verification theorem}

Notice that the BSPDE \eqref{HJBequ_weight} has a $q^*$th-power growth in $v$ in the drift term and the terminal term is $\infty$.  In the first subsection, the existence result for equation \eqref{HJBequ_weight} is proved with the method of truncations, and in the second subsection, we derived the verification theorem as well as the uniqueness. In contrast to the methods in \cite{GraeweHorstQui13,Horst-Qiu-Zhang-14} for BSPDEs on the whole space with quadratic growth, the main difference comes from the treatments of the general $q^*$th-power growth, the bounded gradient estimates, the finer space for solutions, and the nontrivial domain.

\subsection{Existence of a pair of strong solutions to HJB equation \eqref{ValueFunction_u}}
To prove the existence of the strong solution of HJB equation \eqref{ValueFunction_u}, equivalently we can prove the existence for  the weighted equation \eqref{HJBequ_weight}. We first consider the following BSPDE \eqref{HJBequ_weight} with a finite (truncated) terminal condition: for each $M\in\bN^+$,
\begin{equation}\label{HJBequ_truncate}
\left\{
\begin{split}
-dv^M_t(y)&=\left[\alpha D^2v^M+\sigma^*D\zeta^M+\lambda\theta-\frac{|v^M|^{q^*}}{(q^*-1)|\theta\eta|^{q^*-1}}
-\mu(\cZ) v^M
+f(t,y,Dv^M,v^M,\zeta^M)
\right.
\\
&\quad \quad
\left. 
+\int_{\mathcal{Z}}\frac{\theta\gamma(\cdot,z)v^M}{( |\theta\gamma(\cdot,z)|^{q^*-1}+|v^M|^{q^*-1})^{q-1}}\,\mu(dz)   \right](t,y)\,dt
-\zeta^M_t(y)\,dW_t,\\
&\quad \quad \forall (t,y)\in[0,T)\times \cD,\\
D v^M_t(a) &=0, \quad \forall t\in[0,T],\\
v^M_T(y)&=M\theta(y),\quad \forall y\in {\cD}.
\end{split}
\right.
\end{equation}

We will first discuss the existence of strong solutions of the above equation and then derive the existence for BSPDE \eqref{HJBequ_weight} by taking $M\rrow \infty.$

\begin{prop}\label{exis_truncated_equ}
Let Assumption \ref{AssHJB} hold. BSPDE \eqref{HJBequ_truncate} admits a unique strong solution $(v^M,\zeta^M)$ in $\cM^1$. Moreover, $Dv^M\in \sS^{\infty}_{w,\sF}(0,T;L^{\infty}(\cD))$ and the sequence $\{v^M\}_{M\in \bN^+}$ is increasing as $M$ tends to infinity.
\end{prop}

\begin{proof}
Consider the following equation with the $q^*$th-power  term truncated by $N$ in \eqref{HJBequ_truncate}.  That is to use $\Big|\big(\theta^{-1}|v^{M,N}|\big)\wedge N\Big|^{q^*-1}v^{M,N}$ instead of $\theta^{-1}\big|v^{M,N}\big|^{q^*}$.
\begin{equation}\label{HJBequ_truncate_N}
\left\{
\begin{split}
-dv^{M,N}_t(y)
&=\Big[\alpha D^2v^{M,N}+\sigma D\zeta^{M,N}+\lambda\theta
-\frac{\big|\big(\theta^{-1}|v^{M,N}|\big)\wedge N\big|^{q^*-1}}{(q^*-1)|\eta|^{q^*-1}}|v^{M,N}|
-\mu(\cZ) v^{M,N}
\\
&\quad\quad
	+\int_{\mathcal{Z}}\frac{\theta \gamma(\cdot,z) v^{M,N}}{(|\theta\gamma(\cdot,z)|^{q^*-1}+|v^{M,N}|^{q^*-1})^{q-1}}\,\mu(dz)
	\\
&\quad\quad
		+f(t,y,Dv^{M,N},v^{M,N},\zeta^{M,N})\Big](t,y)\,dt
	-\zeta^{M,N}_t(y)\,dW_t,\\
Dv^{M,N}_t(a)&=0,\quad \forall t\in[0,T],\\
v^{M,N}_T(y)&=M\theta(y). 
\end{split}
\right.
\end{equation}
We may check that Assumption \ref{AssLip}  is satisfied. By Theorem \ref{exi_uni_Lip}, the above BSPDE admits a unique strong solution $(v^{M,N},\zeta^{M,N})\in\cM^1$ with 
\begin{align*}
\esssup_{(\omega\in\Omega)} \sup_{t\in[0,T]} \|v^{M,N}_t\|_{H^{1}(\cD)}
+\|v^{M,N}\|_{\sM^2_{\sF}(0,T;H^2(\cD))} + \|\zeta^{M,N}\|_{\sM^{2}_{\sF}(0,T; H^1(\cD))}
<\infty.
\end{align*}

Put
$(v', \zeta')=(Dv^{M,N},D\zeta^{M,N})$. Straightforward computations indicate that $( v', \zeta')$ is a weak solution of the following BSPDE with Dirichlet boundary conditions:
\begin{equation}\label{HJBequ_truncate_N-D}
\left\{
\begin{split}
-d v'_t(y)&=\left[D(\alpha D v')+D(\sigma \zeta')
-\mu(\cZ) v'
+D\left( f(t,y,Dv^{M,N},v^{M,N},\zeta^{M,N})\right)
		\right. \\
&\quad
	\left.
	+D\left(\lambda\theta
	-\frac{\big|\big(\theta^{-1}|v^{M,N}|\big)\wedge N\big|^{q^*-1}}{(q^*-1)|\eta|^{q^*-1}} \theta^{-1}|v^{M,N}| \cdot \theta   
	\right.\right.\\
 &\quad \quad \quad\left.\left.
 	+\int_{\mathcal{Z}}\frac{\theta \gamma(\cdot,z) v^{M,N}}{(|\theta\gamma(\cdot,z)|^{q^*-1}+|v^{M,N}|^{q^*-1})^{q-1}}\mu(dz)
		\right)
		\right](t,y)\,dt
	-\zeta'_t(y)\,dW_t,\\
v'_t(a)&=0,\quad \forall t\in[0,T],\\
v'_T(y)&=MD\theta(y),
\end{split}
\right.
\end{equation}
where 
\begin{align}
&D\left( f(t,y,Dv^{M,N},v^{M,N},\zeta^{M,N})\right)
\nonumber
\\
&= D\left(  \left[\beta+4(y-a) \alpha \theta\right]Dv^{M,N}
+2\theta\left[   \alpha     +(y-a)\beta\right]v^{M,N}
+2(y-a)\theta\sigma^*\zeta^{M,N}  \right)
\nonumber
\\
&=D\left[\beta+4(y-a) \alpha \theta\right] Dv^{M,N}
+\left[\beta+4(y-a) \alpha \theta\right] D^2 v^{M,N}
+ D(2\theta\left[   \alpha     +(y-a)\beta\right]) v^{M,N}
\nonumber
\\
&\quad
+ 2\theta\left[   \alpha     +(y-a)\beta\right]  Dv^{M,N}
+D(2(y-a)\theta\sigma^*)\zeta^{M,N} 
+2(y-a)\theta\sigma^*D \zeta^{M,N}.
\label{est-Df}
\end{align}
Recalling $\theta(y)=(1+(y-a)^2)^{-1}$, we have $(y-a)\theta(y)$ and all the derivatives of $\theta$ are lying in $L^{\infty}(\bR)\cap L^{r}(\bR)$ for all $r>1$, which together with Assumption \ref{AssHJB}   and Sobolev's embedding theorem indicates that Assumption \ref{ass-mp} holds. Indeed, when checking Assumption \ref{ass-mp} with $n=1$, we have $p>\max\{n+2,2+4/n\}=6$ and $\frac{p(n+2)}{p+2+n}>\frac{6(1+2)}{6+2+1}=2$, and as (we may easily check)  $D(2(y-a)\theta\sigma^*) \in \sM^r_{\sF}(0,T;L^r(\cD)) \cap \sS^{\infty}_{\sF}(0,T;L^{\infty}(\cD)$ for all $r>1$, an immediate consequence of H$\ddot{\text{o}}$lder's inequality gives
$$ D(2(y-a)\theta\sigma^*)\zeta^{M,N} \in \sM_{\sF}^{3}(0,T;L^3(\cD)) \cap \sM_{\sF}^{2}(0,T;L^2(\cD)). $$
All the other terms in \eqref{est-Df} follows in a similar way. In addition, we have obviously
\begin{align*}
  &\bigg( \lambda\theta
	-\frac{\big|\big(\theta^{-1}|v^{M,N}|\big)\wedge N\big|^{q^*-1}}{(q^*-1)|\eta|^{q^*-1}} \theta^{-1}|v^{M,N}| \cdot \theta   
	\\
 &\quad \quad \quad
 	+\int_{\mathcal{Z}}\frac{\theta \gamma(\cdot,z) v^{M,N}}{(|\theta\gamma(\cdot,z)|^{q^*-1}+|v^{M,N}|^{q^*-1})^{q-1}}\mu(dz) \bigg)
\in \sM_{\sF}^r(0,T;L^r(\cD)), 
\end{align*}
for all $r\in(1,\infty)$.
Therefore, we may apply Proposition \ref{MP-general-domain-quasilinear-BSPDE} and obtain that  
$$ v'= Dv^{M,N}\in \sS^{\infty}_{w,\sF}(0,T;L^{\infty}(\cD)) .$$

Then, we may apply the comparison theorem in Theorem \ref{comparison-thm} and obtain
$$\bar v^{M,0}\geq v^{M,0}\geq v^{M,1}\geq v^{M,2}\geq\cdots\geq 0,$$
where $(\bar v^{M,0},0)$ with $\theta^{-1}\bar v^{M,0}=M+E\left[\int_t^T\lambda(s)\,ds\big| \sF_t\right]$ is the strong solution to BSPDE \eqref{HJBequ_truncate_N} when $ N=0$ and $\gamma \equiv \infty$. Therefore, for any positive integer $N$, 
$$0\leq\theta^{-1}v^{M,N}\leq M+E\left[\int_t^T\lambda(s)\,ds\big| \sF_t\right] \leq M+(T-t)\Lambda.$$  
If we let $N$ be sufficiently large, then 
$$(\theta^{-1}|v^{M,N}|)\wedge N=\theta^{-1}|v^{M,N}|,$$
and BSPDE \eqref{HJBequ_truncate_N} turns out to be the same as equation \eqref{HJBequ_truncate}. That is, for any $M$, there is $N$ large enough,  such that 
$$(v^M,\zeta^M)=(v^{M,N},\zeta^{M,N})$$ is the unique strong solution to \eqref{HJBequ_truncate}. 

Also, by the comparison theorem, we have that:
$$v^1\leq v^2\leq \cdots\leq v^M\leq v^{M+1}\leq\cdots, \quad \sP\times dy\text{-almost surely}.$$
\end{proof}

\begin{rmk}\label{rmk-comparison-quadratic}
It is worth noting that in the above proof, the comparison theorem (Theorem \ref{comparison-thm}) is actually applied to BSPDEs with nonlinear growth with respect to $v^M$. This works well because $\theta^{-1}v^M$ is uniformly bounded and BSPDE \eqref{HJBequ_truncate} may be equivalently written as BSPDE \eqref{HJBequ_truncate_N} with sufficiently large $N$. Similar arguments will be omitted in what follows unless stated otherwise.
\end{rmk}

Let $M\rrow \infty$, we can prove the following existence theorem for equation \eqref{ValueFunction_u}.
\begin{thm}\label{thm-existence}
Let Assumption \ref{AssHJB} hold. BSPDE \eqref{ValueFunction_u}  admits a strong solution $(u,\psi)$ such that $ (\theta u,\theta\psi)\in \cM^1([0,\tau]\times \cD)$, $D(\theta u) \in \sS^{\infty}_{w,\sF}(0,\tau;L^{\infty}(\cD))$ for $\tau\in[0,T)$, and 
$$
\frac{c_0 }{(T-t)^{q-1}}
	\leq u_t(y) \leq
		\frac{C_0}{(T-t)^{q-1}},\quad\text{a.s. } \forall\,(t,y)\in [0,T)\times \cD,
$$
where the positive constants $c_0>0$ and $ C_0>0$ depend only on $q,\kappa_0, \Lambda, T$ and $\mu(\cZ)$.
\end{thm}

\begin{proof}
\textbf{Step 1:} By Proposition \ref{exis_truncated_equ}, BSPDE \eqref{HJBequ_truncate} with the terminal value $M\theta$ has a unique strong solution $(v^M,\zeta^M)\in\cM^1$, and the sequence $\{v^M\}$ is increasing and thus admits a limit which we denote by $v$. 

Let $(\underline {v}^M,\underline{\zeta}^M)$ and $(\overline{v}^M,\overline{\zeta}^M)$ be the strong solutions to BSPDE \eqref{HJBequ_truncate} with $(\lambda,\eta,\gamma)$ replaced by $(0,\kappa_0,0)$  and $(\Lambda,\Lambda,\infty)$, respectively. 
Then, it is easy to check that $(\underline{\zeta}^M,\,\overline{\zeta}^M)=(0,\,0)$,
\begin{align*}
\underline{v}^M_t(y)
	&= \left( \frac{(q^*-1)|\kappa_0|^{q^*-1}\mu(\mathcal{Z})}
    		{\Big(1+\frac{(q^*-1)|\kappa_0|^{q^*-1}\mu(\mathcal{Z})}{M^{q^*-1}}\Big)e^{(q^*-1)\mu(\mathcal{Z})(T-t)}-1} \right)^{q-1}\theta(y),
\end{align*}
and $\overline{v}^M_t(y)=\Gamma^M_t \theta(y)$ with $\Gamma^M$ satisfying
\begin{align*}
\begin{cases}
-\frac{d\Gamma_t^M}{dt}=\Lambda -\frac{|\Gamma_t^M|^{q^*}}{(q^*-1)\Lambda^{q*-1}},\quad t\in[0,T),\\
\Gamma_T^M=M.
\end{cases}
\end{align*}
The comparison theorem (see Theorem \ref{comparison-thm} and Remark \ref{rmk-comparison-quadratic}) indicates that 
$\underline{v}^M\leq v^M \leq \overline{v}^M$. Letting $M$ tend to infinity gives the boundedness from below:
\begin{align}
v_t(y)\geq \frac{c_0 \theta(y)}{(T-t)^{q-1}},\quad \forall\, (t,y)\in [0,T)\times \cD, \quad \text{a.s.,} \label{eq-low-bd}
\end{align}
with the constant $c_0>0$ depending on $q,\kappa_0, \Lambda, T$ and $\mu(\cZ)$.
Moreover, there are $M_0\in\bN $ and $\eps\in (0,T)$ such that almost surely,
$$
\overline{v}^M _t(y)
	\geq 
		v^M_t(y)
		>
			\Lambda T \theta(y), \quad \text{for all }M>M_0 \text{ and } (t,y)\in [T-\eps,T]\times \cD.
$$
Notice that 
\begin{align*}
\begin{cases}
-\frac{d\left( \Gamma_t^M-\Lambda (T- t) \right)}{dt} =-\frac{\left| \Gamma_t^M-\Lambda (T-t)+\Lambda (T-t) \right|^{q^*}}{(q^*-1)\Lambda^{q*-1}},\quad t\in[T-\eps,T),\\
\Gamma_T^M=M.
\end{cases}
\end{align*}
Then over the interval $[T-\eps, T]$, whenever $M>M_0$, we have $0\leq  \Gamma_t^M-\Lambda(T-t) \leq Y_t^M$ with $Y_t^M$ satisfying:
\begin{align*}
\begin{cases}
-\frac{dY_t^M}{dt} =-\frac{\left| Y_t^M \right|^{q^*}}{(q^*-1)\Lambda^{q*-1}},\quad t\in[T-\eps,T),\\
Y_T^M=M.
\end{cases}
\end{align*}
In fact, solving for $Y_t^M$ gives
\begin{align*}
Y_t^M= \left(
		\frac{1}{\frac{T-t}{\Lambda^{q^*-1}}
						+ \frac{1}{M^{q^*-1}}   }
				\right)^{\frac{1}{q^*-1}}
	\leq \frac{\Lambda}{(T-t)^{q-1}},\quad t\in [T-\eps,T).
\end{align*}

Therefore, for $M>M_0$ and $t\in [T-\eps,T)$, we have 
\begin{align}
\Gamma_t^M\leq  \frac{\Lambda}{(T-t)^{q-1}} + \Lambda (T-t).
	\label{eq-upper-bd-eps}
\end{align}
On the other hand, over the interval $[0,T-\eps]$, it is obvious that for $M>M_0$,
\begin{align}
\Gamma_t^M \leq \Gamma^M_{T-\eps}+\Lambda (T-\eps-t), \quad \forall\, t\in [0,T-\eps]. \label{eq-uppper-bd-0}
\end{align}
Combining \eqref{eq-upper-bd-eps} and \eqref{eq-uppper-bd-0} yields that for all $M>M_0$,
\begin{align}
\Gamma_t^M \leq
	\frac{C_0}{(T-t)^{q-1}} , \quad \forall\, t\in [0,T), \label{eq-upper-bd}
\end{align}
where $C_0$ depends on $ \Lambda$  and $ T$.  This together with \eqref{eq-low-bd} finally implies that with probability 1,
\begin{align}
\frac{c_0 }{(T-t)^{q-1}}
	\leq \frac{ v_t(y)}{\theta(y)} \leq
		\frac{C_0}{(T-t)^{q-1}},\quad \forall\,(t,y)\in [0,T)\times \cD.
				\label{est-v-limit}
\end{align}

\noindent \textbf{Step 2:} For any $M$, if $(v^M,\zeta^M)$ is the strong solution of BSPDE \eqref{HJBequ_truncate}. Then,  for any $0\leq t\leq \tau<T$, $(\hat{v}^M,\hat{\zeta}^M)=((\tau-t)v^M,(\tau-t)\zeta^M)$ is the strong solution of the following equation over the time interval $[0,\tau]$:
\begin{equation*}
\left\{
\begin{split}
-d\hat{v}^M_t(y)&=\Big[\alpha D^2\hat{v}^M+\sigma D\hat{\zeta}^M+(\tau-t)\lambda \theta
-\frac{\hat{v}^M}{q^*-1} \cdot \left| \frac{v^M}{\eta\theta } \right|^{q^*-1}+v^M -\mu(\cZ) \hat v^M
\\
&\,
	+\int_{\mathcal{Z}}\frac{(\tau-t)\theta \gamma(\cdot,z) \hat v^{M}}{\left(  |(\tau-t)\theta\gamma(\cdot,z)|^{q^*-1}+| \hat v^{M}|^{q^*-1}\right)^{q-1}}\,\mu(dz)
	+f(t,y,D\hat{v}^M,\hat{v}^M,\hat{\zeta}^M)\Big] (t,y)\,dt 
	\\
&\quad\quad
-\hat{\zeta}^M_t(y)\,dW_t,\\
D\hat{v}^M_t (a)&=0,\quad t\in[0,\tau],\\
\hat{v}^M_{\tau}(y)&=0. 
\end{split}
\right.
\end{equation*}
At the same time, let $(\hat{v},\hat{\zeta})\in\cM^1([0,\tau]\times\cD)$ be the strong solution to the following BSPDE:
\begin{equation*}
\left\{
\begin{split}
-d\hat{v}_t(y)&=\Big[\alpha D^2\hat{v}+\sigma D\hat{\zeta}+(\tau-t)\lambda \theta 
-\frac{\hat{v}}{q^*-1} \left| \frac{v }{\eta \theta }\right|^{q^*-1}+v -\mu(\cZ) \hat v
\\
&\quad 
+\int_{\mathcal{Z}}\frac{(\tau-t)\theta \gamma(\cdot,z) \hat v}{\left(  |(\tau-t)\theta\gamma(\cdot,z)|^{q^*-1}+| \hat v|^{q^*-1}\right)^{q-1}}\,\mu(dz)
+f(t,y,D\hat{v},\hat{v},\hat{\zeta})\Big](t,y)\,dt\\
&\quad
 -\hat{\zeta}_t(y)\,dW_t,\\
D\hat{v}&=0, \quad t\in[0,\tau],\\
\hat{v}_{\tau}(y)&=0, 
\end{split}
\right.
\end{equation*}
with $v$ being the limit of $v^M$ in Step 1. Notice that
\begin{align*}
\frac{\hat{v}}{q^*-1} \left| \frac{v }{\eta \theta } \right|^{q^*-1}- \frac{\hat{v}^M}{q^*-1} \left| \frac{v^M}{\eta\theta } \right|^{q^*-1}
=
	\frac{\hat{v}}{q^*-1}  \left(\frac{|v|^{q^*-1} -|v^M|^{q^*-1}}{|\eta \theta |^{q^*-1}} \right)
	+ \frac{\hat v- \hat{v}^M}{q^*-1}  \left| \frac{v^M}{\eta\theta } \right|^{q^*-1},
\end{align*}
with the first term on the RHS converging to zero in $\sM^2_{\sF}(0,\tau;L^2(\cD))$. Then, in view of the boundedness estimate \eqref{est-v-limit} and with similar calculations to the proof of Proposition \ref{priori_est}, we have $$\|(\hat{v}^M-\hat{v},\hat{\zeta}^M-\hat{\zeta})\|_{\cM^1([0,\tau] \times \cD )}\rrow 0, \quad \text{as }M\rrow+\infty.$$
Therefore, $\hat{v}=(\tau-t)v,~\cP\times dy$-a.e.. Also, it is easy to check that $D\hat v \in \sS^{\infty}_{w,\sF}(0,\tau;L^{\infty}(\cD))$.

Then for any $\tau_0$ with $\tau_0<\tau<T$, and for $t\in[0,\tau_0]$,  $v=\frac{\hat{v}}{\tau-t}$ and define $\zeta:=\frac{\hat{\zeta}}{\tau-t}$. Then by the arbitrariness of $\tau$, it is easy to check: (i) $(v,\zeta)\in\cM^1([0,\tau_0]\times \cD)$ and $Dv \in \sS^{\infty}_{w,\sF}(0,\tau_0;L^{\infty}(\cD))$ for each $\tau_0<T$;  (ii) as $\tau_0 \rightarrow T^-$,  $v(\tau_0,y)\rrow +\infty$ a.s. for all $y\in\cD$; (iii)  $(v,\zeta)$ satisfies for each $t\in[0,\tau_0]$,
\begin{equation*}
\begin{split}
v_t(y)&=v_{\tau_0}(y)-\int_t^{\tau_0}\left[\alpha D^2v+\sigma^*D\zeta+\lambda \theta
-\frac{|v|^{q^*}}{(q^*-1)|\eta\theta|^{q^*-1}}-\mu(\cZ)v
++f(s,y,Dv,v,\zeta) \right.\\
& \quad\quad
+
\int_{\mathcal{Z}}\frac{\theta\gamma(\cdot,z)v}{ \left( |\theta\gamma(\cdot,z)|^{q^*-1}+ |v|^{q^*-1}\right)^{q-1}}\,\mu(dz)
\bigg](s,y)\,ds 
+\int_t^{\tau_0}\zeta_s(y)\,dW_s,\quad \bP\times dy\text{-a.e.,}
\end{split}
\end{equation*}
and 
$Dv_t(a) =0$ a.s. $\forall t\in[0,\tau_0]$.  Therefore, $(v,\zeta)$ is the strong solution to BSPDE \eqref{HJBequ_weight}. Equivalently, $(u,\psi)=(\theta^{-1} {v},\theta^{-1} {\zeta})$ is the strong solution to BSPDE \eqref{ValueFunction_u}.
\end{proof}


\subsection{Verification theorem and uniqueness of strong solution to BSPDE \eqref{ValueFunction_u}}\label{verification}
\begin{thm}\label{thm-verification}
  Let Assumption \ref{AssHJB} be satisfied and suppose that $(u,\psi)$ is a strong solution to BSPDE \eqref{ValueFunction_u} that satisfies $D(\theta u) \in \sS^{\infty}_{w,\sF}(0,\tau;L^{\infty}(\cD))$ for $\tau\in[0,T)$,
\begin{equation} \label{ABC}
  	({\theta}u,{\theta}\psi)1_{[0,t]}\in\mathcal{H}^1([0,t]\times\cD),\quad t\in(0,T),
\end{equation}	
and
\begin{equation} \label{eq-thm-u}
    \frac{c_0}{(T-t)^{q-1}}\leq u_t(y) \leq \frac{c_1}{(T-t)^{q-1}},\quad \text{a.s. } \forall (t,y)\in [0,T)\times \cD, 
\end{equation}
with $c_0$ and $c_1$ being two  positive constants. Then, 
\begin{equation*}
    V(t,y,x):=u_t(y)|x|^q,\quad (t,x,y)\in[0,T]\times \bR\times \cD,
\end{equation*}
coincides with the value function of \eqref{value-func}. Moreover, the optimal (feedback) control is given by
\begin{equation} \label{eq-thm-control}
  \left(\xi^{*}_t,\, \rho^*_t(z)\right)= \left(\frac{\left|u_t(y_t)\right|^{q^*-1} x_t}{\left| \eta_t(y_t)\right|^{q^*-1}},\, \frac{\left|u_t(y_t)\right|^{q^*-1}x_{t-}}{\left|\gamma_t(y_t,z)\right|^{q^*-1}+\left|u_t(y_t)\right|^{q^*-1}}   \right),
  \quad \text{for }t\in[0,T).
\end{equation}
\end{thm}

In view of \eqref{ABC}, the regularity of $({\theta}u,{\theta}\psi)$ is not high enough to apply the generalized It\^o-Kunita-Wentzell formula (\cite[Lemma 4.1]{Bayraktar-Qiu_2017}) for the compositions of random fields and reflected SDEs. In fact, to the best of our knowledge, no generalized It\^o-Kunita-Wentzell formula is available for the strong solutions of backward SPDEs like \eqref{ValueFunction_u} even for the cases without singular terminal condition. The proof of Theorem \ref{thm-verification} is instead based on 
the representation relationship between FBSDEs and backward SPDEs in Theorem \ref{thm-ito-wentzell}.

The following result is similar to \cite[Lemma 3.4]{GraeweHorstQui13}, and its proof is omitted. It states that the optimal control lies in the set $\mathscr{A}$ for which the corresponding state process is monotone.    
\begin{lem}\label{lem-admissible-control}
  Given any admissible control pair $(\xi,\rho)\in \cL^q_{\bar{\sF}}(0,T)\times \cL^{q}_{\bar{\sF}}(0,T;L^q(\mathcal{Z}))$, we may find a corresponding admissible control pair $(\hat{\xi},\hat{\rho})\in \cL^q_{\bar{\sF}}(0,T)\times \cL^{q}_{\bar{\sF}}(0,T;L^q(\mathcal{Z}))$ satisfying: \begin{enumerate}[(i)] 
  \item the cost associated to $(\hat\xi,\hat\rho)$ is {no} more than that of $(\xi,\rho)$; 
  \item  the corresponding state process
  $x^{0,x;\hat{\xi},\hat{\rho}}$ is a.s. monotone;
\item it holds that for each $t\in[0,T]$,
  \begin{align}
E\left[\left.\sup_{s\in[t,T]}|x_s^{0,x;\hat{\xi},\hat{\rho}}|^q\right|\bar{\sF}_t\right]
    =
   |x_t^{0,x;\hat{\xi},\hat{\rho}}|^q
    \leq\, C (T-t)^{q-1}E\left[\left.\int_t^T|\hat{\xi}_s|^q\,ds
    \right|\bar{\sF}_t\right], \label{est-lem-adm-control}
  \end{align}
  where  the constant $C>0$ is independent of the initial data {$(0,x)$, terminal time $T$} and the control pair $( \hat{\xi},\hat{\rho})$.
  \end{enumerate}
\end{lem}

The proof of Lemma \ref{lem-admissible-control} is similar to that of \cite[Lemma 3.4]{GraeweHorstQui13}, so we omit it here.
%

\begin{proof}[Proof of Theorem \ref{thm-verification}]
First, we check the admissibility of the control process $(\xi^*,\rho*)$. We compute the state process
\begin{equation*} 
    x_{t}^*:=x
     \exp\left(-\int_0^t\frac{|u_s(y_s^{0,y})|^{q^*-1}}{|\eta_s(y_s^{0,y})|^{q^*-1}}\,ds\right)
    \prod_{0<s\leq t}\left\{1-\int_{Z}\frac{|u_s(y_s^{0,y})|^{q^*-1}}{|\gamma_s(y_s^{0,y},z)|^{q^*-1}+|u_s(y_s^{0,y})|^{q^*-1}}\,\pi(dz,\{s\})  \right\}
    .
\end{equation*}
  Obviously, $x_{\cdot}^*$ is monotonic and as $t\uparrow T$,
  \begin{align*}
  |x_{t}^*|
  \leq 
  	\, \left|x \, \textrm{exp}\left\{-\int_0^t \left| \frac{u_s(y_s^{0,y})}{\eta_s(y_s^{0,y})}\right|^{q^*-1}\,ds\right\}  \right|
  &\leq
  	\,|x|\, \textrm{exp}\left\{-\int_0^t \frac{|c_0|^{q^*-1}}{|\Lambda(T-s)^{q-1}|^{q^*-1}}\,ds \right\}
  \\
  &=|x|\left(\frac{T-t}{T}\right)^{\left|\frac{c_0}{\Lambda}\right|^{q^*-1} }\longrightarrow 0
 \end{align*}
as $t\uparrow T.$   In view of the definition of $(\xi^*,\rho^*)$, we see directly $\rho^*\in \cL^{q}_{\bar{\sF}}(0,T;L^q(\mathcal{Z}))$, and $\xi^*\in \cL^q_{\bar{\sF}}(0,\tau)$ for any $\tau\in(0,T)$. In the next step, we shall further confirm $\xi^*\in \sL^q_{\bar\sF}(0,T;\bR)$.

	Second, we note that  the BSPDE \eqref{ValueFunction_u} is equivalent to the BSPDE \eqref{HJBequ_weight} and that by Sobolev's embedding theorem, $u_t(y)$ is a.s. continuous with respect to $(t,y)\in[0,T)\times\bR^d$. If we restrict the BSPDE \eqref{HJBequ_weight} onto the time interval $[0,\tau]$ with $\tau\in (0,T)$, taking $\theta u_\tau(y)$ as the terminal condition, then the BSPDE \eqref{HJBequ_weight} satisfies the assumptions of Theorem \ref{exi_uni_Lip} on time interval $[0,\tau]$, due to \eqref{eq-thm-u} and the presence of the weight function $\theta$. Thus, the pair $(v,\zeta):=(\theta u,\,\theta \psi)$ turns out to be the unique strong solution  to the following BSPDE:
	\begin{equation}\label{HJBequ_weight-tau}
\left\{
\begin{split}
-dv_t(y)&=\left[\alpha D^2v+\sigma^*D\zeta+\lambda\theta
-\frac{|v|^{q^*}}{(q^*-1)|\eta\theta|^{q^*-1}}
-\mu(\cZ)v 
+f(t,y,Dv,v,\zeta) 
					\right.\\
&\quad\quad
+\int_{\mathcal{Z}}\frac{\theta \gamma(\cdot,z) v}{\left(|\theta\gamma(\cdot,z)|^{q^*-1}+|v|^{q^*-1}\right)^{q-1}}\,\mu(dz)	\bigg](t,y)\,dt
-\zeta_t(y)\,dW_t, \\
D v_t(a) &=0, \quad \text{for }t\in[0,\tau],\\
v_{\tau}(y)&=\theta(y) u_{\tau}(y), \quad y\in\bR.
\end{split}
\right.
\end{equation}

By Theorem \ref{thm-ito-wentzell}, we have following BSDE representation:
\begin{align*}
  &-d(\theta u_t)(y_t^{0,y})\\
  &= \,\bigg[
  -\beta D(\theta u)
  +\lambda \theta
  -\frac{|\theta u|^{q^*}}{(q^*-1)|\eta\theta|^{q^*-1}}
-\mu(\cZ)\theta u
  +f(t,y_t^{0,y},D(\theta u)  , \theta u ,\theta \psi)
  \\
&\quad\quad
+\int_{\mathcal{Z}}\frac{\theta \gamma(\cdot,z) \theta u}{\left(|\theta\gamma(\cdot,z)|^{q^*-1}+|\theta u|^{q^*-1}\right)^{q-1}}\,\mu(dz)
\bigg](t,y_t^{0,y})\,dt
	 \\
&\quad\quad
 -\left(D(\theta u) \bar \sigma\right)(t,y_t^{0,y})\,dB_t
	 -\left( \theta\psi+D(\theta u) \sigma  \right)(t,y_t^{0,y})\,dW_t,\quad t\in [0,T).
\end{align*}
 Applying the standard It\^o formula, we obtain
\begin{align*}
  d\theta^{-1}(y_t^{0,y})
  =\bigg[ \alpha  D^2\theta^{-1}
  +\beta  D\theta^{-1}
  \bigg](t, y_t^{0,y})\,dt
  +\left((D\theta^{-1})\sigma\right)(t,y_t^{0,y})\,dW_t
  +   +\left( (D\theta^{-1})\bar\sigma\right)(t,y_t^{0,y})\,dB_t
\end{align*}
and further,
\begin{align*}
  -d u_t(y_t^{0,y})
  = &\,
  \left[ \lambda
  -\frac{|u|^{q^*}}{(q^*-1)|\eta |^{q^*-1}}
-\mu(\cZ)u
  +\int_{\mathcal{Z}}\frac{ \gamma(\cdot,z) u}{\left(|\gamma(\cdot,z)|^{q^*-1}+|u|^{q^*-1}\right)^{q-1}}\,\mu(dz)
  			\right](t,y_t^{0,y})
  \,dt
  \\
  &\,
     - \left[ Du \bar\sigma   \right](t,y_t^{0,y})\,dB_t
  -\left[ \mu+Du \sigma   \right](t,y_t^{0,y})\,dW_t
  ,\quad t\in [0,T).
\end{align*}
{Then the stochastic differential equation for $u_t(y_t^{0,y})|x_t^{0,x;\xi,\rho}|^q$ follows immediately from an application of the standard It\^o formula again.}
Then, for every $(x,y)\in\bR\times \cD$ and each admissible control $(\xi,\rho)$ which, by Lemma \ref{lem-admissible-control}, drives a monotone process $x^{0,x;\xi,\rho}$, we have
\begin{align}
	&u_t(y_t^{0,y})\big|x_t^{0,x;\xi,\rho}\big|^q- E\left[ \left. 
		u_r(y_r^{0,y})\big|x_r^{0,x;\xi,\rho}\big|^q \right|\bar{\sF}_t\right] 
			\nonumber\\
  	&=E\left[ 
		\int_t^r \bigg(\lambda_s(y_s^{0,y})|x_s^{0,x;\xi,\rho}|^q
		-\frac{|u_s(y_s^{0,y})|^{q^*} |x_s^{0,x;\xi,\rho}|^q}{(q^*-1)|\eta_s(y_s^{0,y})|^{q^*-1}}
		+ q u_s(y_s^{0,y})\left|x_s^{0,x;\xi,\rho}\right|^{q-2}x_s^{0,x;\xi,\rho}  \xi_s \right.
		\nonumber\\
	&\quad\quad
		-\mu(\cZ)u_s(y_s^{0,y})|x_s^{0,x;\xi,\rho}|^q
		+u_s(y_s^{0,y})\int_{\mathcal{Z}}\left( \left|x_s^{0,x;\xi,\rho} -\rho_s(y_s^{0,y},z) \right|^q- \left|x_s^{0,x;\xi,\rho}\right|^q 	\right)\,\mu(dz)   
		\nonumber\\
	 	&\qquad
		\left. \left.
		 +\int_{\mathcal{Z}}\frac{\gamma_s(y_s^{0,y},z) u_s(y_s^{0,y})   |x_s^{0,x;\xi,\rho}|^q}{\left( |\gamma_s(y_s^{0,y},z)|^{q^*-1} +|u_s(y_s^{0,y})|^{q^*-1}\right) ^{q-1}}\mu(dz) \bigg)\,ds
			\right|\bar{\sF}_t\right]
		\nonumber\\
		&
		\leq E\left[ \left.
		 \int_t^r\!\! \left(
		 	 \eta_s(y_s^{0,y})|\xi_s|^q +\lambda_s(y_s^{0,y})\big| x_s^{0,x;\xi,\rho}  \big|^q 
			 +\int_{\mathcal{Z}} \gamma_s(y_s^{0,y},z) |\rho_s(z)|^q\,\mu(dz)
			    \right) ds
			    \right|\bar{\sF}_t\right],\text{ a.s.,} \label{eq1-prf-thm-verif}
	\end{align}
  for all $0\leq t\leq \tau <T$. Notice that, the  control pair $(\xi^*,\rho^*)$ makes the above inequality with equality; in particular, we have
   \begin{align*}
    &u_0(y)|x|^q \geq u_0(y)|x|^q - E\left[ u_{\tau}(y_{\tau}^{0,y})|x_{\tau}^{0,x;\xi^*,\rho^*}|^q \right]
    \\
    &=
    E\left[ \int_0^{\tau} \!\!\left( \eta_s(y_s^{0,y})|\xi^*_s|^q
    +\lambda_s(y_s^{0,y})\big| x_s^{0,x;\xi^*,\rho^*}  \big|^q
    +\int_{\mathcal{Z}} \gamma_s(y_s^{0,y},z) |\rho^*_s(z)|^q\,\mu(dz)
     \right)ds  \right], 
\end{align*}
  which indicates that $\xi^*\in \sL^q_{\bar\sF}(0,T;\bR)$ by letting $\tau$ go to $T$. Thus, the control pair $(\xi^*,\rho^*)$ is admissible.
  
  As $\tau\rightarrow T$, by \eqref{eq-thm-u} and \eqref{est-lem-adm-control}, it turns out that
  \begin{align}
\lim_{{\tau}\uparrow T} \bigg|E\left[\left. u_{\tau}(y_{\tau}^{0,y})\left|x_{\tau}^{0,x;\xi,\rho}\right|^q \right|\bar{\sF}_t\right]\bigg|
	\leq
		\lim_{{\tau}\uparrow T} \frac{c_1}{(T-{\tau})^{q-1}} \, C (T-{\tau})^{q-1}E\left[\left.\int_{\tau}^T\!\!\left|\xi_s\right|^q\,ds
    \right|\bar{\sF}_{t}\right] = 0,\label{esti-001}
  \end{align}
which implies that  
  \begin{align*} \label{eq-thm-prf-1}
    u_0(y)|x |^q
    \leq J(0,x ,y ;\xi,\rho) \quad \text{for any admissible } \quad (\xi,\rho),
  \end{align*}
and that the control $(\xi^*,\rho^*)$ satisfies the above inequality with equality and is thus optimal.
\end{proof}

In view of Theorem \ref{thm-verification} and its proof above, we see that the uniqueness of the strong solution $(u,\psi)$ to BSPDE  \eqref{ValueFunction_u} requires the specific domination relation \eqref{eq-thm-u}. In fact, such a requirement may be dropped as below.

\begin{thm}\label{thm-minimal}
Under Assumption \ref{AssHJB}, for the solution $(u,\psi)$ to BSPDE \eqref{ValueFunction_u} constructed in the proof of Theorem \ref{thm-existence}, if $(\tilde{u},\tilde{\psi})$ is another strong solution of \eqref{ValueFunction_u} satisfying
 $$
(\theta\tilde{u},\theta\tilde{\psi}+\sigma^*D(\theta\tilde{\psi}))\in
  \cH^1([0,t]\times \cD)
  \quad\text{and } 
     \tilde{u},\, D(\theta \tilde{u})\in \sS^{\infty}_{w,\sF}(0,t;L^{\infty}(\cO))
     ,\quad \forall\,t\in(0,T),
 $$
 then a.s. for {all} $t\in [0,T)$,
  $  \tilde{u}_t= u_t $ a.e. in $\bR^d$.
\end{thm}

\begin{proof} 
Denote by $(v^M,\zeta^M)\in \cM^1$ the unique solution to BSPDE \eqref{HJBequ_weight}, and by  $(v,\zeta)$ and $(u,\psi)$, respectively, the strong solutions to BSPDEs \eqref{HJBequ_weight} and \eqref{ValueFunction_u} in the proof of Theorem \ref{thm-existence}. The embedding theorem indicates that $v^M\in \sS_{\sF}^{\infty}(0,T;L^{\infty}(\cD))$ and $\tilde v \in \sS_{\sF}^{\infty}(0,t;L^{\infty}(\cD))$ for all $t\in[0,T)$.  Set
$(\tilde{v},\tilde{\zeta})\,=\,\theta\,(\tilde{u},\tilde{\psi})$, and $K_M= \|v^M\|_{\sS^{\infty}_{\sF}(0,T;L^{\infty}(\cD))}$. 
For each $(t,y)\in [0,T)\times \cD$, define
$$\tau^M=\inf\{s>t;\, \tilde{v}_s(y^{t,y}_s) > K_M  \},$$
and obviously, one has $t\leq \tau^M <T$ a.s. Then the boundedness of $\tilde v$ and $v^M$ on the interval $[t,\tau^M]$ allows us to follow the same proof for Theorem \ref{comparison-thm} with the deterministic time interval $[0,T]$ replaced by a random one $[t,\tau^M]$, and we may arrive at the following comparison relation, 
  \begin{equation}\label{eqf-thm-minimal}
    \tilde{v}_t(y)\geq v^M_t(y)  \text{ a.s.}
  \end{equation}
  Letting $M$ tend to infinity yields that  $\tilde{v}_t(y)\geq v_t(y)$ and  a.s. for all $(t,y)\in [0,T)\times \cD$. In particular, it follows from Theorem \ref{thm-verification} that 
  \begin{align}
  \tilde u_t(y) \geq   \frac{c_0}{(T-t)^{q-1}}, \quad{a.s., }\forall (t,y)\in[0,T)\times \cD. \label{relatn-lower}
  \end{align}

	In view of Theorem \ref{thm-verification}, it remains to verify that $\tilde {u}$ satisfies the condition \eqref{eq-thm-u}. The above arguments have given the lower bound \eqref{relatn-lower}. 
For each $t_0\in [0,T)$, taking $T_0\in [0,T)$ such that $0\leq t_0<T_0<T$, as in Step 1 in the proof of Theorem \ref{thm-existence}, we may set $\hat{v}^M_t(y)=\Gamma^M_t \theta(y)$ with $\Gamma^M$ satisfying
\begin{align*}
\begin{cases}
-d\Gamma_t^M=\Lambda -\frac{|\Gamma_t^M|^{q^*}}{(q^*-1)\Lambda^{q*-1}},\quad t\in[0,T_0),\\
\Gamma_{T_0}^M=M,
\end{cases}
\end{align*}
and it is easy to check that $(\hat{v}^M,0)$ is the unique solution to BSPDE \eqref{HJBequ_weight} over the time interval $[0,T_0]$ instead of $[0,T]$. Let $M> \|\tilde v\|_{\sS_{\sF}^{\infty}(0,T_0;L^{\infty}(\cD))}$. Then, over the time interval $[0,T_0]$,  both $\tilde v$ and $\hat v^M$ are bounded, and we may apply Theorem \ref{comparison-thm} as commented in Remark \ref{rmk-comparison-quadratic} and obtain the relation:
$$
\tilde v_{t_0}(y)\leq \hat v^M_{t_0}(y) \leq \frac{C_0 \theta(y)}{(T_0-t_0)^{q-1}}, \quad\text{a.s. }\forall\, y\in \cD,
$$
where the second inequality follows the same to the relation \eqref{eq-upper-bd} with the constant $C_0$ being independent of $(M,T_0)$. Letting $T_0$ tend to $T$ and taking into account the arbitrariness of $t_0$ finally yield that
$$
\tilde u_{t}(y)\leq\frac{C_0}{(T-t)^{q-1}}, \quad\text{a.s. }\forall \, (t,y)\in [0,T)\times \cD,
$$
which is the desired upper bound and together with the lower bound \eqref{relatn-lower}  implies the  uniqueness as a  consequence of Theorem \ref{thm-verification}.
\end{proof}

\begin{appendix}
\section{Generalized It\^o-Wentzell formula by Krylov \cite{Krylov_09} and a corollary}

For an arbitrary domain $\Pi$ in some Euclidean space, let ${C}^\infty_c(\Pi)$ be the class of infinitely differentiable functions with compact support in $\Pi$. Denote  by $\mathscr{D}$ the space of real-valued Schwartz distributions on
$C_c^{\infty}:=C^{\infty}_{c}(\bR^n)$.  By $\mathfrak{D}$ we denote the set of all $\mathscr{D}$-valued functions defined on
$\Omega\times [0,T]$ such that, for any $u\in \mathfrak{D}$ and $\phi\in C_c^{\infty}$, the
function $\langle u,\,\phi\rangle$ is $\sP$-measurable.

For $p=1,2$ we denote by $\mathfrak{D}^p$ the totality of $u\in\mathfrak{D}$ such that for any $R_1,R_2\in(0,\infty)$ and $\phi\in C_c^{\infty}$, we have
$$
\int_0^{R_2} \sup_{|x|\leq R_1} |\langle u(t,\cdot),\phi(\cdot-x)\rangle|^p \,dt<\infty \quad \text{a.s.}
$$

For $u,f,g\in \mathfrak{D}$, we say that the equality
\begin{equation}\label{eq 1}
du(t,x)=f(t,x)\,dt+g(t,x)\,dW_t, \quad t\in [0,T],
\end{equation}
holds in the sense of distributions if $f   \in \mathfrak{D}^1$, $g   \in \mathfrak{D}^2$ and for any $\phi\in C_c^{\infty}$ with probability one we have for all $t\in[0,T]$
$$
\langle u(t,\cdot),\,\phi\rangle=
\langle u(0,\cdot),\,\phi\rangle+\int_0^t\langle f(s,\cdot),\,\phi\rangle\,ds+\int_0^t\langle g(s,\cdot)\,dW_s,\,\phi\rangle.
$$

Let $x_t$ be an $\bR^n$-valued predictable process of the following form
$$
x_t=\int_0^tb_s\,ds+\int_0^t\beta_s\,dW_s +\int_0^t \bar \beta_s\,dB_s,
$$
where $b$, $\beta$ and $\bar\beta$ are $\bar\sF_t$-adapted processes such that for all $\omega\in\Omega$ and $s\in[0,T]$,
we have
$$
{\rm tr } (\beta_s\beta^*_s+\bar\beta_s\bar\beta^*_s) <\infty \quad \hbox { \rm and } \quad \int_0^T [|b_t|+{\rm tr } (\beta_t\beta^*_t+\bar\beta_s\bar\beta^*_s) ]\,dt<\infty.
$$
\begin{thm}[Theorem 1 of \cite{Krylov_09}]\label{Ito-Wentzell}
  Assume that \eqref{eq 1} holds in the sense of distribution and define
  $$
  v(t,x):=u(t,x+x_t).
  $$
  Then we have
  \begin{equation*}
    \begin{split}
      dv(t,x)
      =&\bigg(
      f(t,x+x_t)+\text{tr}\left\{ \frac{1}{2}(\beta_t\beta^*_t+\bar\beta_t\bar\beta^*_t)D^2v(t,x)+Dg(t,x+x_t) \beta_t\right\}
      +b^*_tDv(t,x)\bigg)\,dt\\
      &+\left( g(t,x+x_t)+Dv(t,x)\beta_t \right)\,dW_t
      +Dv(t,x)\bar \beta_t\,dB_t
      ,\quad t\in[0,T]
    \end{split}
  \end{equation*}
  holds in the sense of distributions.
\end{thm}

Let $(L_t)_{t\geq 0}$ be a $\sP$-measurable continuous bounded variation process satisfying $L_0=0$. In Theorem \ref{Ito-Wentzell}, for each $\phi\in C_c^{\infty}$, set
\begin{align*}
u^\phi(t,y)=\langle u(t,y+x_t+\cdot),\,\phi \rangle, \quad \text{for }y\in\bR^n \text{ and }t\geq 0.
\end{align*}
Then, we have for each $y\in\bR^n$ and $t\in[0,T]$,
\begin{align*}
      du^{\phi}(t,y)
      &=\bigg(
       \text{tr}\left\{ \frac{1}{2}(\beta_t\beta^*_t+\bar\beta_t\bar\beta_t^*)D^2u^{\phi}(t,y)+ \langle Dg(t,y+x_t+\cdot)\beta^*_t,\, \phi\rangle\right\}
      +b^*_tDu^{\phi}(t,y)
      \\
      &   \quad
      +  \langle f(t,y+x_t+\cdot),\,\phi\rangle   \bigg)\,dt
      +\left( \langle g(t,y+x_t+\cdot),\,\phi\rangle
      + Du^{\phi}(t,y)\beta_t\right)\,dW_t
      \\
      &\quad
      + Du^{\phi}(t,y)\bar \beta_t\,dB_t, \quad \text{a.s.}
\end{align*}
Notice that both the drift and diffusion terms of the above equation for $u^{\phi}(t,y)$ are smooth with respect to $y\in\bR^n$. A straightforward generalization of \cite[Lemma 4.1]{Bayraktar-Qiu_2017} gives the representation for the composition $u^{\phi}(t,L_t)$ with the following SDE:
\begin{align*}
      du^{\phi}(t,L_t) 
      &=\bigg(
	\text{tr}\left\{ \frac{1}{2}(\beta_t\beta^*_t+\bar\beta_t\bar\beta^*_t) 
       D^2u^{\phi}(t,L_t)+ \langle Dg(t,L_t+x_t+\cdot)\beta_t,\, \phi\rangle\right\}
       +     b^*_tDu^{\phi}(t,L_t)\bigg)\,dt
       \\
      &
      \quad + \langle f(t,L_t+x_t+\cdot),\,\phi\rangle    
            +Du^{\phi}(t,L_t)\,dL_t
      + Du^{\phi}(t,L_t)\bar\beta_t\,dB_t \\
      &
      \quad 
      +\left( \langle g(t,L_t+x_t+\cdot),\,\phi\rangle
      + Du^{\phi}(t,L_t)\beta_t\right)\,dW_t,
       \quad \text{a.s.}
\end{align*}
 In view of the arbitrariness of $\phi$, we have actually arrived at the following assertions.
 
 \begin{cor}\label{Ito-Wentzell-cor}
 In Theorem \ref{Ito-Wentzell}, let $(y_t)_{t\geq 0}$ be an $\bR^n$-valued predictable process of the following form
$$
y_t=\int_0^tb_s\,ds+\int_0^t\beta_s\,dW_s+\int_0^t\bar\beta_s\,dB_s+L_t,
$$
with $(L_t)_{t\geq 0}$ being a $\sP$-measurable continuous bounded variation process satisfying $L_0=0$. If we define 
$$
\Phi(t,y):=u(t,y+y_t),\quad t\geq 0,
$$
then the following equation
  \begin{equation*}
    \begin{split}
      d\Phi(t,y)
      =&\bigg(
      \text{tr}\left\{ \frac{1}{2}(\beta_t\beta^*_t+\bar\beta_t\bar\beta_t^*)
      D^2\Phi(t,y)+ Dg(t,y_t+y)\beta_t\right\}
      +b^*_tD\Phi(t,y)
      +f(t,y+y_t)
      \bigg)\,dt\\
      &
      +D\Phi(t,y) \,dL_t
      +\left( g(t,y+y_t)+D\Phi(t,y)\beta_t \right)\,dW_t
      +D\Phi(t,y)\bar\beta_t \,dB_t
      ,\quad t\in[0,T]
    \end{split}
  \end{equation*}
  holds in the sense of distribution.
 \end{cor}


\section{A maximum principle for weak solutions of quasi-linear backward SPDEs in general domains}
Let $\mathcal{O}\subset \mathbb{R}^n$ be a general nonempty domain that may be unbounded.
We consider the following quasi-linear backward SPDE:
\begin{equation}\label{quasilinear-BSPDE}
\left\{\begin{array}{ll}
\begin{split}
-du(t,x)&=[\partial _j(a^{ij}\partial _iu(t,x)+\sigma ^{jr}v^r(t,x))
+f(t,x,u(t,x),\nabla u(t,x),v(t,x))\\
&+\nabla \cdot g(t,x,u(t,x),\nabla u(t,x),v(t,x))]\,dt
-v^r(t,x)\,dW^r_t,
\quad (t,x)\in Q,\\
u(T,x)&=G(x),~~~x\in \mathcal{O},
\end{split}
\end{array}
\right.
\end{equation}
where $Q=[0,T]\times \cO$, general Dirichlet boundary conditions are endowed and the summation is enforced by convention.


\begin{ass}\label{ass-mp}
\begin{enumerate}[(1)]
	\item The random functions
\[
	g(\cdot,\cdot,\cdot,X,Y,Z):\Omega \times [0,T]\times \mathcal{O}\rightarrow \mathbb{R}^n \quad  \text{and}\quad f(\cdot,\cdot,	\cdot,X,Y,Z):\Omega \times [0,T]\times \mathcal{O}\rightarrow \mathbb{R}
\]
are $\sP\otimes \mathcal{B}(\mathcal{O})$-measurable for any $(X,Y,Z)\in \mathbb{R}\times\mathbb{R}^n\times\mathbb{R}^d$ and there exist positive constants $L$, $\kappa$ and $\beta$ such that for each $(X_i,Y_i,Z_i)\in \mathbb{R}\times\mathbb{R}^n\times\mathbb{R}^d$, $i=1,2$,
$$|g(\cdot,\cdot,\cdot,X_1,Y_1,Z_1)-g(\cdot,\cdot,\cdot,X_2,Y_2,Z_2)|\leq L|X_1-X_2|+\frac{\kappa}{2}|Y_1-Y_2|+\sqrt{\beta}|Z_1-Z_2|$$
and
$$|f(\cdot,\cdot,\cdot,X_1,Y_1,Z_1)-f(\cdot,\cdot,\cdot,X_2,Y_2,Z_2)|\leq L(|X_1-X_2|+|Y_1-Y_2|+|Z_1-Z_2|).$$

	\item  The coefficients $a$ and $\sigma$ are $\mathbb{F}\otimes \mathcal{B}(\mathcal{O})$-measurable and there exist positive constants $\varrho >1$, $\lambda$ and $\Lambda$ such that for each $\eta \in \mathbb{R}^n$ and $(\omega,t,x)\in\Omega\times[0,T]\times\mathcal{O}$,
\begin{align*}
	\lambda |\eta|^2\leq (2a^{ij}(\omega ,t,x)-\varrho \sigma ^{ir}\sigma ^{jr}(\omega,t,x))\eta ^i\eta ^j &\leq \Lambda |\eta|^2\\
	|a(\omega ,t,x)|+|\sigma (\omega ,t,x)| &\leq \Lambda,
\end{align*}
and
$$\lambda -\kappa -\varrho '\beta >0 ~with~\varrho ':=\frac{\varrho}{\varrho -1}.$$

	\item 
The terminal value satisfies $G\in L^{\infty}(\Omega , \mathcal{F}_T, L^2(\mathcal{O}))\cap L^{\infty}(\Omega\times\mathcal{O})$ and for some $p>\max\{n+2,2+4/n\}$, one has
\begin{align*}
	g_0 & :=g(\cdot,\cdot,\cdot,0,0,0)\in  \sM^2_{\sF}(0,T;L^2(\cO)) \cap \sM^p_{\sF}(0,T;L^p(\cO))\\
	f_0 & :=f(\cdot,\cdot,\cdot,0,0,0)\in
			\sM^2_{\sF}(0,T;L^2(\cO)) \cap \sM^{\frac{p(n+2)}{p+n+2}}_{\sF}(0,T;L^{\frac{p(n+2)}{p+n+2}}(\cO)).
\end{align*}

\item  The function $x \mapsto g(\cdot,\cdot,\cdot,x,0,0)$ is uniformly Lipschitz continuous in norm:
\begin{align*}
	\|g(\cdot,\cdot,\cdot,X_1,0,0)-g(\cdot,\cdot,\cdot,X_2,0,0)\|_{\sM^p_{\sF}(0,T;L^p(\cO))}\leq L|X_1-X_2|;\\
    \|g(\cdot,\cdot,\cdot,X_1,0,0)-g(\cdot,\cdot,\cdot,X_2,0,0)\|_{\sM^2_{\sF}(0,T;L^2(\cO))}\leq L|X_1-X_2|.
\end{align*}
\end{enumerate}
\end{ass}

\begin{defn}\label{definition3.1}
The pair $(u,v)\in \left(\sS^2_{\sF}(0,T;L^2(\cO))\cap \sL^2_{\sF}(0,T;H^{1,2}(\cO))\right) \times \sL^2_{\sF}(0,T;L^2(\cO))$ is called a weak solution to backward SPDE \eqref{quasilinear-BSPDE} if
 BSPDE \eqref{quasilinear-BSPDE} holds in the weak sense, i.e., for each $\varphi \in \mathcal{C}_c^{\infty}(\mathbb{R}^+)\otimes \mathcal{C}_c^{\infty}(\mathcal{O})$, we have
\begin{align*}
& \langle u(t,\cdot),\varphi (t,\cdot)\rangle \\
= &
\langle G(\cdot),\varphi(T,\cdot)\rangle
 -\int_t^T\left\{\langle u(s,\cdot),\partial _s\varphi (s,\cdot)\rangle+\langle\partial _j\varphi (s,\cdot),a^{ij}(s,\cdot)\partial _iu(s,\cdot) +\sigma ^{jr}v^r(s,\cdot)\rangle\right\}ds \\
&+\int_t^T\left[\langle f(s,\cdot,u(s,\cdot),\nabla u(s,\cdot),v(s,\cdot)),\varphi (s,\cdot)\rangle
-\langle g^j(s,\cdot,u(s,\cdot),\nabla u(s,\cdot),v(s,\cdot)),\partial_j \varphi (s,\cdot)\rangle\right]ds\\
& -\int_t^T\langle\varphi (s,\cdot),v^r(s,\cdot)dW_s^r\rangle,\quad\text{a.s.}
\end{align*}
\end{defn}

\begin{prop}\label{MP-general-domain-quasilinear-BSPDE}


	If Assumption \ref{ass-mp} holds and $(u,v)$ is a weak solution of backward SPDE \eqref{quasilinear-BSPDE}, then we have $u^{\pm}\in \sS^{\infty}_{w,\sF}(0,T;L^{\infty}(\cO))$ with
\begin{equation}\label{MP-general-domain-quasilinear-BSPDE-1}
\begin{split}
&\textrm{esssup}_{\omega\in\Omega} 
\sup_{t\in[0,T]} \esssup_{x\in\cO} u^{\pm}(\omega,t,x) \\
\leq \, & C\left(\textrm{esssup}_{(\omega,t,x)\in\Omega \times \partial _p\mathcal{O}_t}u^{\pm}+A(f_0^{\pm},g_0 )\right)
\end{split}
\end{equation}
where $C$ depends on $\lambda$, $\kappa$, $\beta$, $L$, $\varrho$, $T$, $p$ and $n$, while $A(f_0^{\pm},g_0 )$ is expressed in terms of some quantities related to the coefficients $f_0$ and $g_0$.
\end{prop}
\begin{rmk}
 Proposition \eqref{MP-general-domain-quasilinear-BSPDE} follows straightforwardly from \cite[Lemma 4.4]{fu2017maximum}. The only difference lies in the fact that we replace the norm $\textrm{esssup}_{(\omega,t,x)\in\Omega \times \mathcal{O}_t}u^{\pm}$ by $\|u^{\pm}\|_{\sS^{\infty}_{w,\sF}(0,T;L^{\infty}(\cO))}$ in the estimate \eqref{MP-general-domain-quasilinear-BSPDE-1}. In fact, if we look into the norms used in \cite[Theorem 3.2]{fu2017maximum}, the norm of space $\mathcal{V}_2(Q)$ could be equivalently set as
 $$
 \|u\|_{\mathcal V_2(Q)}:=\left(\textrm{esssup}_{\omega\in\Omega} \sup_{t\in[0,T]}  \|u(\omega,t,\cdot)\|_{L^2(\cO)}^2+ \|Du\|_{0,2;Q}^2   \right)^{1/2}, \quad \text{for }u\in \mathcal{V}_2(Q),
 $$ 
 because of the time-continuity of the weak solutions, and the iterations in \cite[Proof of Theorem 4.1, pages 317--324]{fu2017maximum} lead to some constant $K\geq 0$ such that  $\|(u-K_{\pm})^{\pm}\|_{\mathcal V_2(Q)}=0$ which yields the fact $u^{\pm}\in \sS^{\infty}_{w,\sF}(0,T;L^{\infty}(\cO))$ as well as the estimate \eqref{MP-general-domain-quasilinear-BSPDE-1}.
 \end{rmk}

\end{appendix}

\bibliographystyle{siam}

\end{document}